\documentclass[reqno]{amsart}
\makeatletter
\def\BState{\State\hskip-\ALG@thistlm}
\makeatother
\usepackage{enumerate}
\usepackage{color}
\usepackage{url}

\newcommand{\newinf}{\mathop{\mathrm{inf}\vphantom{\mathrm{sup}}}}

\usepackage[margin=1.5in]{geometry}

\usepackage{stmaryrd}

\usepackage{amsthm}
\usepackage{amsmath,amsfonts,amstext,amsbsy,mathrsfs,amscd,exscale}
\usepackage{mathtools}
\usepackage{amssymb}
\usepackage{commath}
\usepackage{bigints}
\usepackage{verbatim}
\usepackage{romanbar}
\usepackage{algorithm2e}
\usepackage{MnSymbol}
\usepackage{graphics, epstopdf}

\newtheorem{thm}{Theorem}[section]

\newtheorem{prop}{Proposition}[section]
\newtheorem{cor}{Corollary}[section]
\newtheorem{remark}{Remark}[section]

\newtheorem{lemma}{Lemma}[section]

\newcommand{\vc}{\mathbf{c}}

\newcommand{\vx}{\mathbf{x}}
\newcommand{\vy}{\mathbf{y}}
\newcommand{\vz}{\mathbf{z}}

\newcommand{\vv}{\mathbf{v}}
\newcommand{\vw}{\mathbf{w}}
\newcommand{\vf}{\mathbf{f}}
\newcommand{\vn}{\mathbf{n}}

\newcommand{\vvarphi}{\boldsymbol{\varphi}}
\newcommand{\vphi}{\boldsymbol{\phi}}
\newcommand{\vnu}{\boldsymbol{\nu}}

\newcommand{\jump}[1]{[ #1 ]} 
\newcommand{\avrg}[1]{\left\{ #1 \right\}} 
\newcommand{\Th}{\mathcal{T}_h} 
\newcommand{\Eh}{\mathcal{E}_h}
\newcommand{\E}{\mathcal{E}} 
\newcommand{\K}{T} 
\newcommand{\V}{\mathbb{V}} 
 
\newcommand{\h}{ {\rm h}} 
\newcommand{\Gh}{\Gamma_h}
\newcommand{\bz}{\boldsymbol{0}}

\newcommand{\I}{\normalfont{\Romanbar{1}}} 
\newcommand{\II}{\normalfont{\Romanbar{2}}} 
\newcommand{\g}{g} 
\newcommand{\tr}{{\rm tr}} 
\newcommand{\di}{\mathop{\rm div}\nolimits} 
\newcommand{\supp}{\mathop{\rm supp}\nolimits} 
\newcommand{\A}{\mathbb{A}}
\newcommand{\veps}{\varepsilon}
\newcommand{\wb}{\overline}
\newcommand{\wt}{\widetilde}
\newcommand{\wh}{\widehat}

\def\restriction#1#2{\mathchoice
	{\setbox1\hbox{${\displaystyle #1}_{\scriptstyle #2}$}
		\restrictionaux{#1}{#2}}
	{\setbox1\hbox{${\textstyle #1}_{\scriptstyle #2}$}
		\restrictionaux{#1}{#2}}
	{\setbox1\hbox{${\scriptstyle #1}_{\scriptscriptstyle #2}$}
		\restrictionaux{#1}{#2}}
	{\setbox1\hbox{${\scriptscriptstyle #1}_{\scriptscriptstyle #2}$}
		\restrictionaux{#1}{#2}}}
\def\restrictionaux#1#2{{#1\,\smash{\vrule height .8\ht1 depth .85\dp1}}_{\,#2}}

\numberwithin{equation}{section}

\title[LDG method for large deformations of prestrained plates]{Numerical analysis of the LDG method for large deformations of prestrained plates}
\author{Andrea Bonito}\thanks{(Andrea Bonito) Department of Mathematics, Texas A\&M University, College Station, TX 77843, USA. Email: bonito@tamu.edu.}
\author{Diane Guignard}\thanks{(Diane Guignard) Department of Mathematics and Statistics, University of Ottawa,
Ottawa, ON K1N 6N5, Canada. Email: dguignar@uottawa.ca.}
\author{Ricardo H. Nochetto}\thanks{(Ricardo H. Nochetto) Department of Mathematics and Institute for Physical Science
		and Technology, University of Maryland,
		College Park, Maryland 20742, USA. Email: rhn@umd.edu.}
\author{Shuo Yang}\thanks{(Shuo Yang) Yanqi Lake Beijing Institute of Mathematical Sciences and Applications, Beijing 101408, China, and Yau Mathematical Sciences Center, Tsinghua University, Beijing 100084, China. Email: shuoyang@bimsa.cn.}
\date{\today}

\begin{document}
\maketitle

\begin{abstract} 
A local discontinuous Galerkin (LDG) method for approximating large deformations of prestrained plates is introduced and tested on several insightful numerical examples in \cite{BGNY2020_comp}. This paper presents a numerical analysis of this LDG method, focusing on the \emph{free boundary} case. The problem consists of minimizing a fourth order bending energy subject to a nonlinear and nonconvex metric constraint. The energy is discretized using LDG and a discrete gradient flow is used for computing discrete minimizers. We first show $\Gamma$-convergence of the discrete energy to the continuous one. Then we prove that the discrete gradient flow decreases the energy at each step and computes discrete minimizers with control of the metric constraint defect. We also present a numerical scheme for initialization of the gradient flow, and discuss the conditional stability of it.
\end{abstract}

\medskip
\keywords{\textbf{Keywords}: Prestrained materials; metric constraint; local discontinuous Galerkin; reconstructed Hessian; discrete gradient flow; free boundary conditions}
\section{Introduction}
Prestrained materials can develop internal stresses at rest, deform out of plane even without an external force, and exhibit nontrivial 3d shapes. This is a rich area of research with numerous applications for instance in nematic glasses \cite{modes2010disclination,modes2010gaussian}, natural growth of soft tissues \cite{goriely2005differential,yavari2010geometric} and manufactured polymer gels \cite{kim2012thermally,klein2007shaping,wu2013three}. 

Starting from 3d hyperelasticity, a geometric nonlinear and dimensionally reduced energy for isotropic prestrained plates where bending was the chief mechanism for deformation was proposed in \cite{efrati2009} and derived rigorously via $\Gamma$-convergence in \cite{lewicka2016}. The bending energy reads
\begin{equation}\label{E:reduced-bending}
  E(\vy) = \frac{\mu}{12}\int_{\Omega}\left|g^{-\frac{1}{2}}\II[\vy] g^{-\frac{1}{2}}\right|^2+\frac{\lambda}{2\mu+\lambda}\tr\left(g^{-\frac{1}{2}}\II[\vy] g^{-\frac{1}{2}} \right)^2
\end{equation}
and is subject to the nonlinear and nonconvex metric constraint
\begin{equation}\label{E:metric-constraint}
\I[\vy](\vx) = g(\vx) \quad \text{a.e. in } \Omega,
\end{equation}
where $\vy:\Omega\to\mathbb{R}^3$ is the deformation of the midplane $\Omega\subset\mathbb{R}^2$, $g:\Omega\to\mathbb{R}^{2\times 2}$ is a given symmetric positive definite matrix, and $\lambda$ and $\mu$ are Lam\'e parameters of the material. Hereafter, $\I[\vy]$ and $\II[\vy]$ denote respectively the first and second fundamental forms of the deformed plate $\vy(\Omega)$, namely
\begin{equation}\label{fund_forms}
\I[\vy]:=\nabla\vy^T\nabla\vy \quad \mbox{and} \quad \II[\vy]:=-\nabla\vnu^T\nabla\vy=(\partial_{ij}\vy\cdot\vnu)_{i,j=1}^2,
\end{equation}
where $\vnu:= \frac{\partial_1 \vy \times \partial_2 \vy}{|\partial_1 \vy \times \partial_2 \vy|}$ is the unit normal vector to the surface $\vy(\Omega)$. 
Moreover, $|\cdot|$ stands for the Frobenius norm.
Given an immersible metric $g$, our goal is to construct a deformation $\vy$ that minimizes \eqref{E:reduced-bending} subject to \eqref{E:metric-constraint}.

In \cite{BGNY2020_comp}, we depart from the 3d elastic energy of prestrained plates based on the Saint-Venant Kirchhoff energy density for classical isotropic materials and derive formally the 2d energy \eqref{E:reduced-bending} with a modified Kirchhoff-Love assumption. 
In the special case $g=I_2$ with $I_2$ the $2\times 2$ identity matrix (i.e., when $\vy$ is an isometry), thanks to the relations \cite{bartels2013,bonito2015,bonito2018} 
\begin{equation}\label{E:bilaplacian}
|\II[\vy]|=|D^2\vy|=|\Delta\vy|=\tr(\II[\vy]),
\end{equation}
\eqref{E:reduced-bending} and \eqref{E:metric-constraint} reduce to the nonlinear Kirchhoff plate model:
  minimize the energy
\begin{equation}\label{E:g=1}
E(\vy)=\frac{\alpha}{2}\int_{\Omega}|D^2\vy|^2, \qquad \alpha:=\frac{\mu(\mu+\lambda)}{3(2\mu+\lambda)},
\end{equation}
subject to the isometry constraint $\nabla\vy^T\nabla\vy=I_2$ a.e. in $\Omega$.
A formal derivation of \eqref{E:g=1} can be traced back to Kirchhoff in 1850, and an ansatz-free rigorous derivation was carried out in the seminal work of Friesecke, James, and M\"{u}ller \cite{frie2002b} via $\Gamma$-convergence.

\subsection{Problem statement}
The presence of the highly nonlinear quantity $\II[\vy]$ in the energy \eqref{E:reduced-bending} is an obstacle to the design of efficient numerical algorithms. Since for a general $g$ the relation \eqref{E:bilaplacian} does not hold, the energy given in \eqref{E:reduced-bending} cannot be reduced to \eqref{E:g=1}. However, thanks to Proposition \ref{prop:alt_energy} given in the appendix, we get an equivalent formulation by replacing the second fundamental form $\II[\vy]$ in \eqref{E:reduced-bending} by the Hessian $D^2\vy$ of the deformation.

We intend to study the approximation of the following constrained minimization problem
\begin{equation} \label{prob:min_Eg}
\min_{\vy\in\A} E(\vy),
\end{equation}
where
\begin{equation} \label{def:Eg_D2y}
E(\vy) :=\frac{\mu}{12}\sum_{m=1}^3\int_{\Omega}\left|\g^{-\frac{1}{2}}D^2y_m\g^{-\frac{1}{2}}\right|^2+\frac{\lambda}{2\mu+\lambda}\tr\left(\g^{-\frac{1}{2}}D^2y_m \g^{-\frac{1}{2}}\right)^2
\end{equation}
with $\vy = (y_m)_{m=1}^3$, and where the set of {\it admissible} functions is
\begin{equation} \label{def:admiss}
\A:=\left\{\vy\in [H^2(\Omega)]^3: \, \nabla\vy^T\nabla\vy=\g \quad \mbox{a.e. in } \Omega\right\}.
\end{equation}
Throughout this work, we assume that $g$ is {\it immersible} in $\mathbb{R}^3$, i.e., the admissible set $\A$ is not empty. 
In addition, we assume that $g\in[H^1(\Omega)\cap L^{\infty}(\Omega)]^{2\times 2}$. 

\subsection{Numerical methods}
In the case $g=I_2$, \eqref{E:g=1} is discretized with Kirchhoff finite elements in \cite{bartels2013} and symmetric interior penalty discontinuous Galerkin (SIPG) methods in \cite{bonito2018}.  For bilayer plates with isometry constraint, discretizations relying on Kirchhoff finite elements and on SIPG methods are proposed in \cite{bonito2015,bonito2017} and \cite{bonito2020discontinuous}, respectively. In our previous computational work \cite{BGNY2020_comp}, we consider \eqref{E:g=1} with a general immersible $g \ne I_2$, introduce a {\it local discontinuous Galerkin} (LDG) approach in which the Hessian $D^2\vy$ is replaced by a reconstructed Hessian $H_h(\vy_h)$, and explore the performance of LDG computationally. The present manuscript provides a mathematical justification of several properties of the algorithms in \cite{BGNY2020_comp}, such as convergence, energy decrease and metric defect control. 
 
For Kirchhoff finite elements, the discrete isometry constraint is imposed at the nodes of the mesh and Dirichlet boundary conditions are incorporated in the admissible set. 
They are based on polynomials with degree $k=3$ and require the computation of a discrete gradient, which may complicate the implementation of the method.

In both the LDG and SIPG approaches, the pointwise metric constraint is relaxed by imposing it on average over the elements, and any prescribed boundary conditions are imposed weakly via the \emph{Nitsche} approach, thereby allowing for more geometric flexibility. Furthermore, the method is well defined for polynomials with degree $k=2$ which are implemented in most standard finite element libraries.
For instance, we refer to the step-82 tutorial program \cite{BG21} for an implementation of the reconstructed Hessian in the \textrm{deal.ii} \cite{bangerth2007} library.
 
Compared to SIPG, LDG is conceptually simpler in that it uses $H_h(\vy_h)$ as a chief constituent of the method. Moreover, contrary to SIPG which requires sufficiently large stabilization parameters for stability, there is no such condition for LDG. Stability of LDG is ensured for any positive stabilization parameters, as proved in Theorem \ref{coercivity:plates} below. We refer to \cite[Section 3.1.1]{BGNY2020_comp} for further comments on the comparison of SIPG and LDG.

The LDG method was originally proposed in \cite{cockburn1998}. Motivated by the lifting and discrete gradient operator introduced in \cite{ern2010,ern2011}, the discrete Hessian
\begin{equation}\label{e:discrete-Hessian}
H_h(\vy_h) := D_h^2 \vy_h - R_h(\jump{\nabla_h\vy_h})+B_h(\jump{\vy_h})
\end{equation}
consists of three parts: the broken Hessian $D_h^2 \vy_h$, the lifting of the jumps of the broken gradient $\jump{\nabla_h\vy_h}$, and the lifting of the jumps $\jump{\vy_h}$ of $\vy_h$ itself; a precise definition is given in \eqref{def:discrHess} below. Lifting operators were initially introduced in \cite{bassi1997} and further analyzed in \cite{brezzi1999,brezzi2000}. It is worth mentioning that similar discrete Hessians are used in \cite{pryer} to study the convergence of dG for the bi-Laplacian and in \cite{bonito2018} to prove the $\Gamma$-convergence for plates with isometry constraint. In the present work, $H_h(\vy_h)$ is an integral part of the numerical method and not a mere theoretical device. A key property of $H_h(\vy_h)$ is consistency with integration by parts which yields its weak convergence in $L^2(\Omega)$ (see Lemma \ref{weak-conv} below).

\subsection{Discrete problem} \label{subsec:discrete_intro}

The LDG counterpart of the energy \eqref{def:Eg_D2y} reads
\begin{equation} \label{e:Eh_intro}
\begin{aligned}
  E_h(\vy_h) := & \frac{\mu}{12} \sum_{m=1}^3\int_{\Omega} \left|\g^{-\frac{1}{2}} \, H_h(y_{h,m}) \, \g^{-\frac{1}{2}} \right|^2
  \\ &+ \frac{\mu\lambda }{12(2\mu+\lambda)} \sum_{m=1}^3\int_{\Omega} \tr\left(\g^{-\frac{1}{2}} \, H_h(y_{h,m}) \, \g^{-\frac{1}{2}}\right)^2
 \\ &+\frac{\gamma_1}{2}\|\h^{-\frac{1}{2}}\jump{\nabla_h\vy_h}\|_{L^2(\Gh^0)}^2+\frac{\gamma_0}{2}\|\h^{-\frac{3}{2}}\jump{\vy_h}\|_{L^2(\Gh^0)}^2,
\end{aligned}
\end{equation}
where $\vy_h=(y_{h,m})_{m=1}^3\in[\V_h^k]^3$ is the discrete deformation over
a shape-regular mesh $\Th$ of $\Omega$ (in the sense given in Subsection \ref{subsec:mesh}), $H_h(\vy_h)$ is given in \eqref{e:discrete-Hessian},
$\gamma_0,\gamma_1>0$ are stabilization parameters, and $\Gh^0$ is the skeleton of $\Th$
defined in \eqref{e:skeleton}. Given a parameter $\veps>0$ so that
$\veps\to0$ as $h\to 0$, the {\it discrete admissible set} is
\begin{equation}\label{e:Ah_intro}
\A_{h,\veps}^k:=\Big\{\vy_h\in [\V^k_h]^3: \quad D_h(\vy_h)\leq \veps\Big\},
\end{equation}
where
\begin{equation} \label{def:PD_intro}
D_h(\vy_h) := \sum_{\K\in\Th}\left|\int_{\K} \nabla\vy_h^T\nabla\vy_h-g\right|
\end{equation}
is the {\it metric defect}, also called \emph{prestrain defect} in what follows. Note that $\A_{h,\veps}^k$ is nonconvex. Finally, the discrete counterpart of \eqref{prob:min_Eg} reads
\begin{equation} \label{prob:min_Eg_h_intro}
\min_{\vy_h\in\A_{h,\veps}^k} E_h(\vy_h).
\end{equation}
This is a nonconvex energy minimization due to \eqref{e:Ah_intro} 
and is discussed in detail in Section~\ref{sec:discrete}.

\subsection{Contributions and outline}

In this article, we analyze the algorithms proposed in \cite{BGNY2020_comp}. Following \cite{bartels2013,bonito2015,bonito2018}, we develop a $\Gamma$-convergence theory and show that (up to a subsequence) the discrete global minimizers of the discrete energy \eqref{prob:min_Eg_h_intro} converge to global minimizers of the continuous counterpart \eqref{prob:min_Eg}. 
We focus on the \emph{free boundary} case (no Dirichlet boundary conditions imposed), which is not considered in previous numerical analysis works on large deformation of plates with metric constraint \cite{bartels2013,bonito2015,bonito2018}. We then examine the discrete $H^2$-gradient flow with linearized metric constraint proposed in \cite{BGNY2020_comp}, and prove that the discrete energy decreases at each step while the metric defect is kept under control. Deformations in the \emph{free boundary} case are defined up to rigid motions which requires the addition of an $L^2$ term in the gradient flow metric. Last but not least, we study the behavior of (a generalization of) the preprocessing algorithm proposed in \cite{BGNY2020_comp} and designed to construct an initial deformation for the main gradient flow with a small prestrain defect. An error analysis is out of reach partly due to poor understanding of the nonconvex constraint \eqref{E:metric-constraint} and lack of characterization of immersible metrics.

The rest of the article is organized as follows. In Section \ref{sec:discrete}, we introduce the (\emph{broken}) finite element spaces and prove preliminary key properties for discrete functions, such as Poincar\'{e}-Friedrichs type inequalities and a compactness result. The discrete Hessian operator $H_h$ is discussed in Subsection \ref{subsec:H_h} together with weak and strong convergence properties, bounds on the lifting operators, and an equivalence relation crucial to prove the coercivity of the discrete energy. In Section \ref{sec:Eh_Dh}, we define the discrete problem briefly introduced in Subsection \ref{subsec:discrete_intro} above and investigate its properties. The proof of $\Gamma$-convergence of the discrete energy to the exact one is the content of Section \ref{sec:Gamma}. In Section \ref{sec:GF}, we recall the gradient flow scheme used in \cite{BGNY2020_comp} to solve the discrete problem, prove its unconditional stability and show how the prestrain defect is controlled throughout the flow.  The preprocessing algorithm is discussed in Section \ref{sec:preprocessing}. The equivalence between the energy \eqref{E:reduced-bending} and \eqref{def:Eg_D2y}, where the second fundamental forms are replaced by Hessians is the subject of Appendix~\ref{a:alternate}. For completeness, we also discuss in Appendix~\ref{sec:Dirichlet_BC} how the theory for \emph{free boundary} conditions can be extended to settings where Dirichlet boundary conditions are imposed on a portion $\Gamma^D\neq \emptyset$ of the boundary $\partial\Omega$ and where the plate is subject to external forces.

The notations $A\lesssim B$ and $A\sim B$ used throughout stand for $A\leq CB$ and $cB\leq A\leq CB$, where $c,C$ are constants independent of the discretization parameters $h,\veps$ and $\tau$.

\section{Discontinuous finite elements} \label{sec:discrete}

\subsection{Subdivisions} \label{subsec:mesh}
Let $\Omega \subset \mathbb R^2$ be a polygonal domain and consider  a sequence $\{\Th\}_{h>0}$ of shape-regular conforming partitions of $\Omega$ made of either triangles or quadrilaterals $\K$ of diameter $h_{\K}\leq h$.
Let $\Eh^0$ be the set of interior edges and $\Gamma_h^0$ be the interior skeleton of $\Th$
\begin{equation}\label{e:skeleton}
  \Gamma_h^0:=\{ \mathbf{x} \in e \, : \, e\in \E_h^0\}.
\end{equation}
For triangles, the family of meshes $\{\Th\}_{h>0}$ is assumed to be regular in the sense of Ciarlet \cite{ciarlet1991}: there exists a constant $\varrho>0$ independent of $h$ such that
\begin{equation}\label{def:regular_triangle}
\frac{h_{\K}}{\rho_{\K}} \le \varrho\qquad \forall\, \K \in \Th,
\end{equation}
where $\rho_\K$ denotes the diameter of the largest ball inscribed in $\K$. For quadrilaterals, we assume that the elements are convex and that the subtriangles obtained by bisecting an element along each of its diagonals satisfy \eqref{def:regular_triangle}.
Under these regularity conditions, there is an invertible affine (resp. bi-affine) mapping $F_{\K}:\widehat\K\to\K$ that maps the reference unit triangle (resp. square) $\widehat\K$ onto $\K$ and which satisfies
\begin{equation}\label{rel:shape_regular}
\|DF_{\K}\|_{L^{\infty}(\hat\K)} \lesssim  h_{\K}, \qquad \|DF_T^{-1}\|_{L^{\infty}(\K)} \lesssim h_{\K}^{-1},
\end{equation}
where $DF_{\K}$ and $DF^{-1}_{\K}$ denote the Jacobian matrices of $F_{\K}$ and $F^{-1}_{\K}$, respectively; see for instance \cite{ciarlet1991,GR86}.
In order to simplify the notation, we use a mesh function $\h$ such that $h_\K \lesssim \restriction{\h}{\K} \lesssim h_\K$  for all $\K \in \Th$ and $h_e \lesssim   \restriction{\h}{e} \lesssim h_e$ for all $e \in \E_h$.

\subsection{Broken spaces}

Let $k\geq 0$ and $\mathbb{P}_k$ (resp. $\mathbb{Q}_k$) be the space of polynomials of total degree at most $k$ (resp. of degree at most $k$ in each variable). 
Each component of the deformation $\vy$ is approximated by functions from the (\textit{broken}) finite element space
\begin{equation} \label{def:Vhk_tri}
\V_h^k:=\left\{v_h\in L^2(\Omega): \,\, \restriction{v_h}{\K}\circ F_{\K}\in\mathbb{P}_k \,\, (\mbox{resp. } \mathbb{Q}_k) \quad \forall\, \K \in\Th \right\}
\end{equation}
when $\Th$ is made of triangles (resp. quadrilaterals).
In view of the energy \eqref{e:Eh_intro}, we require from now on that $k\geq 2$.
Throughout this work, functions with values in $\mathbb R^3$ are written with bold symbols and subindices indicate their components; for example, $y_{h,m} \in \V_h^k$, $m=1,2,3$ are the components of $\vy_h \in [\V_h^k]^3$. The broken gradient of a scalar function $v_h\in\V_h^k$ is given by $\nabla_h v_h$.
We use a similar notation for other piecewise differential operators, for instance $D_h^2 v_h=\nabla_h\nabla_h v_h$ denotes the broken Hessian. For vector-valued functions these operators are computed component-wise.

We now introduce the jump and average operators. To this end, let $\vn_e$ be a unit normal to $e\in\Eh^0$ (the orientation is chosen arbitrarily but is fixed once for all). For $v_h \in \V_h^k$ and $e \in \E_h^0$, let $v_h^{\pm}(\vx):=\lim_{s\rightarrow 0^+}v_h(\vx\pm s\vn_e)$ for any $\vx \in e$, and set
\begin{equation} \label{def:jump-avrg}
\restriction{\jump{v_h}}{e} := v_h^{-}-v_h^+, \qquad \restriction{\avrg{v_h}}{e} := \frac{1}{2}(v_h^{+}+v_h^{-}).
\end{equation}
The jumps and averages of non-scalar functions are computed component-wise.

\subsection{Discrete Poincar\'e-Friedrichs type inequalities and compactness}

We introduce the mesh-dependent bilinear form $\langle\cdot,\cdot\rangle_{H_h^2(\Omega)}$ defined for any $v_h,w_h\in \V_h^k$ by
\begin{equation} \label{def:H2bilinear}
\begin{aligned}
  \langle v_h,w_h\rangle_{H_h^2(\Omega)} := & (D^2_h v_h,D^2_h w_h)_{L^2(\Omega)}
  \\& +(\h^{-1}\jump{\nabla_h v_h},\jump{\nabla_h w_h})_{L^2(\Gh^0)}+(\h^{-3}\jump{v_h},\jump{w_h})_{L^2(\Gh^0)},
\end{aligned}
\end{equation}
where $\Gh^0$ is defined in \eqref{e:skeleton}. Hereafter,
$(\cdot,\cdot)_{L^2(\varpi)}$ denotes the $L^2(\varpi):=L^2(\varpi; d\mathbf{x})$ inner product associate with the Lebesgue measure $d\mathbf{x}$ on $\mathbb R$ or $\mathbb R^2$ depending on whether $\varpi$ is a measurable set of dimension $1$ or $2$. 
We also define 
\begin{equation}\label{e:H2semi}
  |v_h|_{H_h^2(\Omega)}^2:=\langle v_h,v_h\rangle_{H_h^2(\Omega)}
  \qquad\forall \, v_h\in \V_h^k.
\end{equation}
Note the slight abuse of notation as $\langle \cdot,\cdot\rangle_{H_h^2(\Omega)}$ is not a scalar product; $|\cdot|_{H_h^2(\Omega)}$ is just a semi-norm.  For vector-valued functions $\vv_h, \vw_h \in [\V_h^k]^3$, we define
$\langle \vv_h,\vw_h\rangle_{H_h^2(\Omega)}:= \sum_{m=1}^3 \langle v_{h,m},w_{h,m}\rangle_{H_h^2(\Omega)}$ and similarly for $|\vv_h|_{H^2_h(\Omega)}$. 

In contrast to \eqref{def:H2bilinear}, we introduce the following scalar product and norm on $\V_h^k$
\begin{equation} \label{e:H2metric}
(v_h,w_h)_{H_h^2(\Omega)} := \langle v_h, w_h\rangle_{H_h^2(\Omega)}  + (v_h,w_h)_{L^2(\Omega)}, \qquad \|v_h\|_{H_h^2(\Omega)}^2:=(v_h,v_h)_{H_h^2(\Omega)}.
\end{equation}
They are critical to guarantee the unique solvability of the linear system arising in each step of the gradient flow algorithm of Section~\ref{sec:GF} and control of the metric defect \eqref{def:PD_intro}.

To derive discrete Poincar\'e-Friedrichs inequalities, we rely on the smoothing interpolation operator $\Pi_h:\mathbb{E}(\Th):=\prod_{T\in\Th}H^1(T) \to \V_h^k\cap H^1(\Omega)$ defined by $\Pi_h := I_h \circ P_h$, where $I_h$ is the Cl\'ement type interpolant proposed in \cite{bonito2010quasi} and $P_h$ is an element-wise $L^2(T)$ projection onto the restriction $\V_h^k(T)$ of $\V_h^k$ to $T\in \Th$. The domain $\mathbb{E}(\Th)$ of $\Pi_h$ is larger than $\V_h^{k}$ because the smoothing operator shall be employed on functions in $\nabla \V_h^k$. The latter are in general not in $\V_h^{k'}$ for any $k'\geq 0$ when $\V_h^{k}$ is based on quadrilateral elements; see \eqref{def:Vhk_tri}.

Before embarking on the proof of the discrete Poincar\'e-Friedrichs inequalities, we record several properties of $\Pi_h$. For any $v\in \mathbb{E}(\Th)$ we have
\begin{equation}\label{estimate-smoothing-interpolation-2}
\|\Pi_hv\|_{L^2(\Omega)}\lesssim\|v\|_{L^2(\Omega)},
\end{equation}
\begin{equation} \label{eqn:smooth_interp_H1}
\|\nabla \Pi_h v\|_{L^2(\Omega)} + \|\h^{-1}(v-\Pi_h v)\|_{L^2(\Omega)} \lesssim \|\nabla_h v\|_{L^2(\Omega)}+\|\h^{-\frac{1}{2}}\jump{v}\|_{L^2(\Gh^0)},
\end{equation}
and
\begin{equation}\label{estimate-smoothing-interpolation-1}
\|\h^{-1}(\nabla_hv-\nabla\Pi_hv)\|_{L^2(\Omega)}\lesssim\|D^2_hv\|_{L^2(\Omega)}+\|\h^{-\frac12}\jump{\nabla_hv}\|_{L^2(\Gamma^0_h)}+\|\h^{-\frac32}\jump{v}\|_{L^2(\Gamma^0_h)}.
\end{equation}
Estimate \eqref{estimate-smoothing-interpolation-2} follows from the $L^2(\Omega)$ stability of $I_h$ \cite{bonito2010quasi}, estimate \eqref{eqn:smooth_interp_H1} is guaranteed by Lemma 2.1 in \cite{bonito2018}  and similar arguments can be used to derive \eqref{estimate-smoothing-interpolation-1}.

We are now in position to derive the following discrete Poincar\'e-Friedrichs inequalities.
\begin{lemma}[discrete Poincar\'e-Friedrichs inequalities]\label{estimate-smoothing-interpolation}
For any $v \in \mathbb{E}(\Th)$ there holds
\begin{equation}\label{eqn:fp-smoothing}
\big\|v-\strokedint_{\Omega}v\big\|_{L^2(\Omega)} \lesssim\|\nabla_hv\|_{L^2(\Omega)}+\|\h^{-\frac{1}{2}}\jump{v}\|_{L^2(\Gh^0)},
\end{equation}
where $\strokedint_{\Omega}$ stands for the average over $\Omega$.
Moreover, for any $v_h\in \V_h^k$ there holds
\begin{equation} \label{eqn:bound_noBC}
\|\nabla_h v_h\|_{L^2(\Omega)}+ \|\nabla \Pi_h v_h\|_{L^2(\Omega)} \lesssim  \|v_h\|_{L^2(\Omega)}+|v_h|_{H_h^2(\Omega)}.
\end{equation}
\end{lemma}
\begin{proof} We split the proof in several steps.

\smallskip\noindent
{\it Step 1.} Let $v \in \mathbb{E}(\Th)$. The Cauchy-Schwarz inequality and a standard Poincar\'{e}-Friedrichs inequality for $\Pi_hv\in H^1(\Omega)$ yield
\begin{equation*}
\begin{split}
\big\|v-\strokedint_{\Omega}v\big\|_{L^2(\Omega)}&\le\|v-\Pi_hv\|_{L^2(\Omega)}+\big\|\Pi_h v-\strokedint_\Omega\Pi_hv\big\|_{L^2(\Omega)} +\big\|\strokedint_\Omega \left( v-\Pi_hv\right)\big\|_{L^2(\Omega)}\\
&\lesssim\|v-\Pi_hv\|_{L^2(\Omega)}+\|\nabla\Pi_hv\|_{L^2(\Omega)}.
\end{split}
\end{equation*}
Since $\restriction{\h}{T}\lesssim h_T\le {\rm diam}(\Omega)$ for all $\K\in\Th$, the first estimate \eqref{eqn:fp-smoothing} then directly follows from \eqref{eqn:smooth_interp_H1} with a hidden constant that depends on $\Omega$.

\smallskip\noindent
{\it Step 2.} We claim that for any $v_h \in \V_h^k$ there holds 
\begin{equation}\label{estimate-smoothing-interpolation-3}
\big\|\nabla\Pi_h v_h - \strokedint_\Omega \nabla\Pi_h v_h\big\|_{L^2(\Omega)}\lesssim\|D^2_h v_h\|_{L^2(\Omega)}+\|\h^{-\frac12}\jump{\nabla_h v_h}\|_{L^2(\Gamma^0_h)}+\|\h^{-\frac32}\jump{ v_h}\|_{L^2(\Gamma^0_h)}.
\end{equation}
To see this, we first employ \eqref{eqn:fp-smoothing} on each component of $\nabla\Pi_hv_h\in [\mathbb{E}(\Th)]^3$ to write
\begin{equation*}
\big\|\nabla\Pi_hv_h-\strokedint_{\Omega}\nabla\Pi_hv_h\big\|_{L^2(\Omega)}\lesssim\|D^2_h\Pi_hv_h\|_{L^2(\Omega)}+\|\h^{-\frac12}\jump{\nabla\Pi_hv_h}\|_{L^2(\Gamma^0_h)}.
\end{equation*}
Therefore, to obtain \eqref{estimate-smoothing-interpolation-3} it remains to show that
\begin{equation*}
\|D^2_h\Pi_hv_h\|_{L^2(\Omega)}+\|\h^{-\frac12}\jump{\nabla\Pi_hv_h}\|_{L^2(\Gamma^0_h)}\lesssim\|D^2_hv_h\|_{L^2(\Omega)}+\|\h^{-\frac12}\jump{\nabla_hv_h}\|_{L^2(\Gamma^0_h)}+\|\h^{-\frac32}\jump{v_h}\|_{L^2(\Gamma^0_h)},
\end{equation*}
which can be deduced from standard scaling arguments and equivalence of norms on finite dimensional spaces. 
The details are omitted but we refer to the proof of Lemma 6.6 in \cite{bonito2010quasi} for additional information. 

\smallskip\noindent
{\it Step 3.} We write  $\nabla_h v_h = \nabla \Pi_h  v_h + \nabla_h(v_h-\Pi_h v_h)$.
On the one hand, we infer from \eqref{estimate-smoothing-interpolation-1} that
\begin{equation*}
\|\nabla_h v_h-\nabla\Pi_h v_h\|_{L^2(\Omega)}\lesssim\|D^2_h v_h\|_{L^2(\Omega)}+\|\h^{-\frac12}\jump{\nabla_h v_h}\|_{L^2(\Gamma^0_h)}+\|\h^{-\frac32}\jump{v_h}\|_{L^2(\Gamma^0_h)}.
\end{equation*}
On the other hand, \eqref{estimate-smoothing-interpolation-3} implies
\begin{equation*}\label{intermediate-estimate}
\|\nabla\Pi_h v_h\|_{L^2(\Omega)}\lesssim\|D^2_h v_h\|_{L^2(\Omega)}+\|\h^{-\frac12}\jump{\nabla_h v_h}\|_{L^2(\Gamma^0_h)}+\|\h^{-\frac32}\jump{ v_h}\|_{L^2(\Gamma^0_h)}+\big\| \strokedint_{\Omega}\nabla\Pi_h v_h\big\|_{L^2(\Omega)},
\end{equation*}
whence, recalling the definition \eqref{e:H2semi} of $|\cdot|_{H^2_h(\Omega)}$, we arrive at
\begin{equation}\label{intermediate-estimate-2}
\|\nabla_h v_h\|_{L^2(\Omega)}+\|\nabla \Pi_h v_h\|_{L^2(\Omega)} \lesssim | v_h |_{H^2_h(\Omega)}+\big\| \strokedint_{\Omega}\nabla\Pi_h v_h\big\|_{L^2(\Omega)}.
\end{equation}
It remains to estimate $\| \strokedint_{\Omega}\nabla\Pi_h v_h\|_{L^2(\Omega)}$.
If $\vn_{\partial \Omega}$ denotes the outward unit normal vector to $\partial\Omega$, integrating by parts
\begin{equation*}
\int_{\Omega}\nabla\Pi_h v_h=\int_{\partial\Omega}(\Pi_h v_h)\vn_{\partial\Omega}
\end{equation*}
and combining Cauchy-Schwarz, trace and Young's inequalities, we obtain for any $\epsilon>0$ 
\begin{equation*}
\big\|\strokedint_{\Omega}\nabla\Pi_h v_h\big\|_{L^2(\Omega)}  \lesssim \|\Pi_h v_h\|_{L^2(\partial\Omega)}  \lesssim   \epsilon\|\nabla\Pi_h v_h\|_{L^2(\Omega)}+ \epsilon^{-1}\|\Pi_h v_h\|_{L^2(\Omega)}.
\end{equation*}
The desired estimate \eqref{eqn:bound_noBC} follows from the $L^2(\Omega)$ stability \eqref{estimate-smoothing-interpolation-2} of $\Pi_h$ and upon choosing $\epsilon$ sufficiently small so that the term $\epsilon\|\nabla\Pi_h v_h\|_{L^2(\Omega)}$ in the above estimate can be absorbed in the left-hand side of \eqref{intermediate-estimate-2}. This ends the proof.
\end{proof} 

We end this subsection with a compactness result for discrete balls
$$
\big\{ v_h \in  \V_h^k \ : \ \| \nabla_h v_h \|_{L^2(\Omega)} + | v_h |_{H^2_h(\Omega)} \lesssim 1 \big\}.
$$
As we shall see, the discrete energy \eqref{e:Eh_intro} provides control of the $|\cdot|_{H^2_h(\Omega)}$ semi-norm while a uniform bound for the broken $H^1(\Omega)$ semi-norm is guaranteed for functions in the discrete admissible set \eqref{e:Ah_intro}. 

\begin{lemma}[compactness]\label{l:compactness}
Assume that $\{v_h\}_{h>0}\subset \V_h^k$ is a sequence such that 
\begin{equation}\label{e:unif-bound}
 \| \nabla_h v_h \|_{L^2(\Omega)} + | v_h |_{H^2_h(\Omega)} \lesssim 1.
\end{equation}
Then there exists  $\wb v\in H^2(\Omega)$ with $\strokedint_{\Omega} \wb v = 0$ such that (up to a subsequence) $\wb v_h := v_h - \strokedint_{\Omega} v_h \to \bar v$ in $L^2(\Omega)$ and $\nabla_h \wb v_h\to\nabla \bar v$ in $[L^2(\Omega)]^3$ as $h \to0$.
\end{lemma}
\begin{proof}
We let $c_h:= \strokedint_{\Omega} v_h$ and invoke the Poincar\'e-Friedrichs inequality \eqref{eqn:fp-smoothing} to write
\begin{equation*}
\|v_h - c_h \|_{L^2(\Omega)}^2 \lesssim\|\nabla_h v_h\|_{L^2(\Omega)}^2+\|\h^{-\frac{1}{2}}\jump{v_h}\|_{L^2(\Gh^0)}^2 \lesssim 1.
\end{equation*}
This, together with the uniform boundedness assumption, implies
\begin{equation}\label{e:unifboundcompact}
\| v_h - c_h\|_{L^2(\Omega)} + \| \nabla_h v_h \|_{L^2(\Omega)} + \| D^2_h v_h\|_{L^2(\Omega)} \lesssim 1.
\end{equation}

With this bound being established, the rest of the proof readily follows step~1 - step~3 of Proposition 5.1 in \cite{bonito2018};
it is therefore only sketched here. 
The uniform bound \eqref{e:unifboundcompact} guarantees that $\wb v_h=v_h-c_h$ converges weakly (up to a subsequence) in $L^2(\Omega)$ to some $\wb v$.  
Setting $\wb z_h := \Pi_h v_h - \strokedint_{\Omega} \Pi_h v_h \in \V_h^k \cap H^1(\Omega)$, we invoke the Poincar\'e-Friedrichs inequality \eqref{eqn:fp-smoothing} coupled with the $H^1(\Omega)$ stability \eqref{eqn:smooth_interp_H1} of $\Pi_h$ to deduce that $\wb z_h$ is uniformly bounded in $H^1(\Omega)$. 
As a consequence, $\wb z_h$ converges strongly (up to a subsequence) in $L^2(\Omega)$ to some $\wb z \in H^1(\Omega)$. 
To show that $\wb v=\wb z$, we note that $ \| (v_h - c_h) - \wb z_h \|_{L^2(\Omega)} \to 0$ as $h \to 0$ because of
the interpolation property \eqref{estimate-smoothing-interpolation-1}, Poincar\'e-Friedrichs inequality \eqref{eqn:fp-smoothing} and the uniform boundedness \eqref{e:unif-bound}; hence,
$$
\|( v_h - c_h) -  \wb z \|_{L^2(\Omega)} \leq \|( v_h - c_h) - \wb z_h \|_{L^2(\Omega)} + \|\wb z_h - \wb z \|_{L^2(\Omega)} \to 0 \qquad \mbox{as } h\to 0.
$$
The uniqueness of weak limits guarantees that $\wb v = \wb z$ and thus $v_h-c_h$ strongly converges (up to a subsequence) in $L^2(\Omega)$ to $\wb v \in H^1(\Omega)$.
Repeating this argument for $\nabla_h \wb v_h$ yields that $\nabla_h \wb v_h$ converges strongly in $[L^2(\Omega)]^3$ (up to a subsequence) to $\nabla \wb v$ and $\wb v \in H^2(\Omega)$.
\end{proof}

\subsection{Discrete Hessian} \label{subsec:H_h}

The LDG approximation \eqref{e:Eh_intro} of the elastic energy \eqref{def:Eg_D2y} relies on the discrete approximation $H_h(\vy_h)\in\left[L^2(\Omega)\right]^{3\times 2\times 2}$ of the Hessian $D^2\vy$ introduced in \eqref{e:discrete-Hessian}; we now give a precise definition. The convergence of $E_h$ naturally depends on the convergence of the discrete Hessian towards $D^2 \vy$. The piecewise Hessian $D_h^2 \vy_h$ alone does not contain enough information and cannot be used as discrete approximation $H_h(\vy_h)$. In fact, the jumps of $\vy_h$ and $\nabla_h \vy_h$ must be accounted for. This is the purpose of the lifting operators.

Let $l_1,l_2$ be two non-negative integers, and consider the {\it local lifting operators} $r_e:[L^2(e)]^2\rightarrow[\V_h^{l_1}]^{2\times 2}$ and $b_e:L^2(e)\rightarrow[\V_h^{l_2}]^{2\times 2}$ defined for $e\in\Eh^0$ by
\begin{gather} 
r_e(\vphi) \in [\V_h^{l_1}]^{2\times 2}: \quad
\int_{\Omega}r_e(\vphi):\tau_h = \int_e\avrg{\tau_h}\vn_e\cdot\vphi \qquad \forall\, \tau_h\in [\V_h^{l_1}]^{2\times 2}\label{def:lift_re},
\\
b_e(\phi) \in [\V_h^{l_2}]^{2\times 2}: \quad
\int_{\Omega} b_e(\phi):\tau_h = \int_e\avrg{\di \tau_h}\cdot\vn_e\phi \qquad \forall\, \tau_h\in [\V_h^{l_2}]^{2\times 2} \label{def:lift_be};
\end{gather}
note that $\supp(r_e(\vphi))=\supp(b_e(\phi))=\omega_e$, the union of the two elements sharing $e$.
These lifting operators are extended to $[L^2(e)]^{3\times 2}=\left[[L^2(e)]^2\right]^3$ and $[L^2(e)]^3$, respectively, by component-wise applications. 
The {\it global lifting operators} are then given by
\begin{equation}\label{E:global-lifting}
R_h := \sum_{e\in\Eh^0} r_e : [L^2(\Gh^0)]^2 \rightarrow [\V_h^{l_1}]^{2\times 2},
\qquad
B_h := \sum_{e\in\Eh^0} b_e : L^2(\Gh^0) \rightarrow [\V_h^{l_2}]^{2\times 2}.
\end{equation} 
It is worth mentioning that this construction is simpler than the one in \cite{bonito2018} for quadrilaterals, which had to be defined on $D^2_h\V_h$ for the method to match the interior penalty discretization.
As a consequence, the weak convergence of the discrete Hessian considered in \cite{bonito2018} towards its corresponding exact Hessian requires a restrictive assumption on the sequence of subdivisions (see Proposition 4.3 in \cite{bonito2018}). This restriction is not needed in Lemma 2.4 below.

The following estimates for $R_h$ and $B_h$ can be found e.g. in \cite{brezzi2000,bonito2018}.
\begin{lemma}[stability of lifting operators]\label{L2bound-lifting}
For any $v_h\in\V_h^k$ and for any $l_1,l_2\ge 0$ we have 
\begin{equation*}
\|R_h(\jump{\nabla_h v_h})\|_{L^2(\Omega)}\lesssim \|\h^{-\frac12}\jump{\nabla_h v_h}\|_{L^2(\Gh^0)}, \qquad \|B_h(\jump{v_h})\|_{L^2(\Omega)}\lesssim \|\h^{-\frac32}\jump{v_h}\|_{L^2(\Gh^0)}.
\end{equation*}
\end{lemma}

As anticipated in \eqref{e:discrete-Hessian}, the discrete Hessian operator $H_h:\V_h^k\rightarrow\left[L^2(\Omega)\right]^{2\times 2}$ is defined as
\begin{equation} \label{def:discrHess}
H_h(v_h) := D_h^2 v_h -R_h(\jump{\nabla_hv_h})+B_h(\jump{v_h}).
\end{equation}
The definition \eqref{E:global-lifting} together with integration by parts of $D_h^2 v_h$ yields weak convergence of $H_h(v_h)$. It also gives strong convergence of $H_h(v_h)$ provided $v_h$ is the Lagrange interpolant of a given $v\in H^2(\Omega)$. These results are stated in Lemmas \ref{weak-conv}
and \ref{strong-conv} below, whose proofs are postponed to Appendix \ref{sec:weak_strong}.
Such results are rather standard for the discrete gradient operator in the LDG context \cite{ern2011} and extend to the discrete Hessian \cite{pryer,bonito2018}.

\begin{lemma}[weak convergence of $H_h$]\label{weak-conv}
	Let $\{ v_h\}_{h>0} \subset \V_h^k$ be such that $|v_h|_{H_h^2(\Omega)}\lesssim 1$ uniformly in $h$ and $v_h\to v$ in $L^2(\Omega)$ as $h\rightarrow 0$ for some $v \in H^2(\Omega)$. 
	Then for any polynomial degrees $l_1,l_2\ge0$ of $R_h$ and $B_h$, we have
	\begin{equation} \label{eqn:weakH}
	H_h(v_h)\rightharpoonup D^2 v \quad \mbox{in } \left[L^2(\Omega)\right]^{2\times 2} \qquad \mbox{as } h\rightarrow 0.
	\end{equation} 
\end{lemma} 

\begin{lemma}[strong convergence of $H_h$]\label{strong-conv}
	Let $v\in H^2(\Omega)$ and let $v_h:=\mathcal{I}_h^k v \in\V^k_h \cap H^1(\Omega)$ be the Lagrange interpolant of $v$. Then for any polynomial degrees $l_1,l_2\ge0$ of $R_h$ and $B_h$, we have the following strong convergences in  $[L^2(\Omega)]^{2\times 2}$
	$$
D^2_h v_h \to D^2 v, \qquad R_h(\jump{\nabla_h v_h}) \to 0, \qquad
	B_h(\jump{ v_h}) \to 0 \qquad  \textrm{as }h\to 0.
	$$
	In particular, 
	\begin{equation} \label{eqn:H_strong}
	H_h(v_h)\to D^2 v \quad \mbox{in } \,\, [L^2(\Omega)]^{2\times 2} \qquad  \textrm{as }h\to 0.
	\end{equation} 	 
\end{lemma} 

We end this subsection by showing that the quantity $ \| H_h(\cdot)\|_{L^2(\Omega)}+\| \h^{-\frac12} \jump{\nabla_h \cdot}\|_{L^2(\Gamma_h^0)} + \| \h^{-\frac32} \jump{\cdot}\|_{L^2(\Gamma_h^0)}$ is equivalent to the $|\cdot|_{H^2_h(\Omega)}$ semi-norm.
The definition \eqref{def:discrHess} of $H_h(v_h)$ and Lemma~\ref{L2bound-lifting}
(stability of lifting operators) readily imply
\begin{equation}\label{eqn:Eh<=H2}
\int_{\Omega}|H_h(v_h)|^2  +\gamma_1\sum_{e\in\Eh^0}\int_e\h^{-1}|\jump{\nabla_h v_h}|^2 +\gamma_0\sum_{e\in\Eh^0}\int_e\h^{-3}|\jump{v_h}|^2 \lesssim |v_h|_{H_h^2(\Omega)}^2.
\end{equation}
We now prove the converse and trickier inequality.
\begin{lemma}[discrete $H^2$ semi-norm equivalence]\label{coercivity:hessian}
For any stabilization parameters $\gamma_1,\gamma_0>0$ there exists a constant $C(\gamma_0,\gamma_1)>0$ such that for any $v_h\in \V_h^k$ and any polynomial degrees $l_1,l_2\ge 0$ there holds
\begin{equation}\label{eqn:ineq_h2}
C(\gamma_0,\gamma_1)|v_h|_{H_h^2(\Omega)}^2\le \int_{\Omega}|H_h(v_h)|^2  +\gamma_1\sum_{e\in\Eh^0}\int_e\h^{-1}|\jump{\nabla_h v_h}|^2 +\gamma_0\sum_{e\in\Eh^0}\int_e\h^{-3}|\jump{v_h}|^2.
\end{equation}
Moreover, the constant $C(\gamma_0,\gamma_1)$ tends to $0$ when $\gamma_0$ or $\gamma_1$ tends to $0$.
\end{lemma}

\begin{proof}
We define
\begin{equation} \label{eqn:splitH}
I_1:= \int_{\Omega}|H_h(v_h)|^2, \qquad I_2:= \gamma_1\sum_{e\in\Eh^0}\int_e\h^{-1}|\jump{\nabla_hv_h}|^2 + \gamma_0\sum_{e\in\Eh^0}\int_e\h^{-3}|\jump{v_h}|^2,
\end{equation}
and prove a lower bound for $I_1$. The definition \eqref{def:discrHess} of the discrete Hessian yields
\begin{align*}
I_1 &= \|D_h^2v_h\|_{L^2(\Omega)}^2+\|B_h(\jump{v_h})-R_h(\jump{\nabla_hv_h})\|_{L^2(\Omega)}^2+2\int_{\Omega}D_h^2v_h:\big(B_h(\jump{v_h})-R_h(\jump{\nabla_h v_h})\big)\\
&\ge (1-\alpha^{-1})\|D_h^2v_h\|_{L^2(\Omega)}^2+(1-\alpha)\|B_h(\jump{v_h})-R_h(\jump{\nabla_hv_h})\|_{L^2(\Omega)}^2,
\end{align*}
where we used Young's inequality with $\alpha>1$. Note that Lemma~\ref{L2bound-lifting} (stability of the lifting operators) guarantees the existence of a constant $C$ independent of $h$ such that
\begin{equation*}
\|B_h(\jump{v_h})-R_h(\jump{\nabla_hv_h})\|_{L^2(\Omega)}^2\le C \|\h^{-\frac32}\jump{v_h}\|_{L^2(\Gh^0)}^2+C\|\h^{-\frac12}\jump{\nabla_hv_h}\|_{L^2(\Gh^0)}^2.
\end{equation*}
Returning to the estimate for $I_1$, we thus arrive at
\begin{align*}
I_1 \ge (1-\alpha^{-1})\|D_h^2v_h\|_{L^2(\Omega)}^2+(1-\alpha)C\|\h^{-\frac32}\jump{v_h}\|_{L^2(\Gh^0)}^2+(1-\alpha)C\|\h^{-\frac12}\jump{\nabla_hv_h}\|_{L^2(\Gh^0)}^2.
\end{align*}
Combining this with $I_2$, we deduce that
\begin{equation*}
I_1+I_2 \ge \min\Big\{1-\alpha^{-1},(1-\alpha)C+\gamma_0,(1-\alpha)C+\gamma_1\Big\}|v_h|_{H_h^2(\Omega)}.
\end{equation*}
Therefore, recalling that $\gamma_0,\gamma_1>0$, the assertion \eqref{eqn:ineq_h2} holds with
\begin{equation} \label{def:coercivity_cst}
C(\gamma_0,\gamma_1):=\min\left\{1-\alpha^{-1},(1-\alpha)C+\gamma_0,(1-\alpha)C+\gamma_1\right\}
\end{equation}
upon choosing $1 < \alpha < 1+\min(\gamma_0,\gamma_1)/C$.
\end{proof}

\section{Discrete energy and discrete admissible set} \label{sec:Eh_Dh} 

We now deal with the discrete energy $E_h(\vy_h)$ defined in
\eqref{e:Eh_intro}.
Compared to the exact energy \eqref{def:Eg_D2y}, the Hessians $D^2 y_k$ are replaced by the discrete Hessians $H_h(y_{h,k})$ and stabilization terms with parameters $\gamma_0,\gamma_1>0$ are included. The latter are motivated by the following coercivity result, which holds for any parameters $\gamma_0,\gamma_1>0$. Note that they are not required to be large enough unlike for the interior penalty method \cite{bonito2018}.

\begin{thm}[coercivity of $E_h$]\label{coercivity:plates}
Let $\vy_h\in [\V^k_h]^3$ and let $\gamma_0,\gamma_1>0$. There holds
\begin{equation} \label{eqn:coercivity_free}
|\vy_h|_{H_h^2(\Omega)}^2 \lesssim E_h(\vy_h).
\end{equation}
The hidden constant in the above estimate depends only on $\mu$, $g$, and the constant $C (\gamma_0,\gamma_1)$ that appears in \eqref{eqn:ineq_h2}, and tends to infinity as $\gamma_0$ or $\gamma_1 \to 0$.
\end{thm}

\begin{proof}
Let $R(H_h)$ denote the range of $H_h:[\V^k_h]^3 \rightarrow [L^2(\Omega)]^{3\times 2\times 2}$. Because the metric $g(\vx)$ is SPD for a.e. $\vx \in \Omega$, the quantity $\big(\int_{\Omega}|g^{-\frac{1}{2}}\cdot g^{-\frac{1}{2}}|^2\big)^{\frac{1}{2}}:R(H_h)\to\mathbb{R}$ is a norm in the finite dimensional space $R(H_h)$ and is thus equivalent to $\| \, |\cdot | \,\|_{L^2(\Omega)}$, where $|\cdot|$ is the Frobenius norm.
Hence, there exists a constant $C>0$ depending only on $g$ such that 
\begin{equation*}
C \frac{\mu}{12} \|H_h(\vy_h)\|_{L^2(\Omega)}^2 +\frac{\gamma_1}{2}\|\h^{-\frac{1}{2}}\jump{\nabla_h\vy_h}\|_{L^2(\Gh^0)}^2+\frac{\gamma_0}{2}\|\h^{-\frac{3}{2}}\jump{\vy_h}\|_{L^2(\Gh^0)}^2 \leq E_h(\vy_h).
\end{equation*}
Lemma \ref{coercivity:hessian} (discrete $H^2$ semi-norm equivalence) implies
the desired estimate.
\end{proof}

We now discuss the approximation $\A_{h,\veps}^k$ of the admissible set $\A$ defined in \eqref{e:Ah_intro}.
The pointwise metric constraint $\nabla\vy^T\nabla\vy=g$ is too strong to be imposed on a polynomial space.
This leads to the definitions \eqref{def:PD_intro} and \eqref{e:Ah_intro} of the metric defect $D_h(\vy_h)$ and the discrete admissible set $\A_{h,\veps}^k$, namely
\begin{equation*} \label{def:PD}
  D_h(\vy_h) = \sum_{\K\in\Th}\left|\int_{\K} \nabla\vy_h^T\nabla\vy_h-g \,\right|,
  \qquad
  \A_{h,\veps}^k = \Big\{\vy_h\in [\V^k_h]^3\ :\  D_h(\vy_h)\leq \veps \Big\},
\end{equation*}
for a positive number $\veps$. 
The discrete counterpart of \eqref{prob:min_Eg} finally reads
$
\min_{\vy_h\in\A_{h,\veps}^k} E_h(\vy_h).
$

Recall that by assumption, $g$ is immersible and so $\A\neq\emptyset$. The following lemma guarantees that $\A_{h,\veps}^k$ is not empty provided that $\veps$ is sufficiently large.
\begin{lemma}[$\A_{h,\veps}^k$ is non-empty] \label{lemma:Ah_not_empty}
	Let $\vy\in\A$ and let $\vy_h:=\mathcal{I}_h^k\vy\in [\V^k_h]^3\cap[H^1(\Omega)]^3$ be the Lagrange interpolant of $\vy$. Then there exists a constant $C>0$ depending only on the shape regularity of $\{\Th\}_{h>0}$ and $\Omega$ such that 
	\begin{equation*} 
	D_h(\vy_h)\le Ch\|\vy\|_{H^2(\Omega)}^2.
	\end{equation*}
	In particular, $\vy_h \in \A_{h,\veps}^k$ provided $\veps\ge Ch\|\vy\|_{H^2(\Omega)}^2$.
\end{lemma}
\begin{proof}
	We proceed as in Step 2 of Proposition 5.3 in \cite{bonito2018} and compute
	\begin{equation}\label{e:add_sub_contraint}
	\left(\nabla_h\vy_h^T\nabla_h\vy_h-g\right) - \left(\nabla\vy^T\nabla\vy-g\right)=\nabla_{h}(\vy_h-\vy)^T\nabla_{h}\vy_h+\nabla\vy^T\nabla_h(\vy_h-\vy).
	\end{equation}
	Because $\vy \in \A$, further algebraic manipulation yields
	$$
	\nabla_h\vy_h^T\nabla_h\vy_h-g =\nabla_{h}(\vy_h-\vy)^T\nabla\vy+\nabla\vy^T\nabla_h(\vy_h-\vy)+ \nabla_{h}(\vy_h-\vy)^T\nabla_{h}(\vy_h-\vy)
	$$
        whence, thanks to the interpolation estimate
	$$
	\| \nabla_h (\vy - \vy_h) \|_{L^2(\Omega)} \lesssim h | \vy|_{H^2(\Omega)},
	$$
	we obtain
	\begin{equation*}
	D_h(\vy_h)\leq \| \nabla_h\vy_h^T\nabla_h\vy_h-g \|_{L^1(\Omega)} \lesssim \left( \| \nabla\vy \|_{L^2(\Omega)} + h|\vy|_{H^2(\Omega)}  \right) h|\vy|_{H^2(\Omega)} \lesssim  h\|\vy\|^2_{H^2(\Omega)},
	\end{equation*}
        which is the desired estimate.
	\end{proof}

Lemma~\ref{l:compactness} (compactness) requires sequences uniformly bounded in $|\cdot|_{H^2_h(\Omega)}$ and in the $H^1(\Omega)$ broken semi-norm. The former stems from Theorem~\ref{coercivity:plates} (coercivity of $E_h$) for sequences with bounded energies. For the latter, we resort to the constraint encoded in the discrete admissible set $\A^k_{h,\veps}$.
This is the object of the next lemma.

\begin{lemma}[gradient estimate]\label{bound-nabla}
We have 
\begin{equation}\label{e:gradient_estim}
\|\nabla_h\vy_h\|_{L^2(\Omega)}^2\leq  \sqrt{2}(\veps+\|g\|_{L^1(\Omega)}) \qquad  \forall\, \vy_h\in\A^k_{h,\veps}.
\end{equation}
\end{lemma}
\begin{proof}
It suffices to note that for any $T\in\Th$
\begin{align}\label{eqn:transpose-estimate}
\frac{1}{2}\left(\int_{\K}|\nabla\vy_h|^2\right)^2 \leq \sum_{i=1}^2\left(\int_{\K}|\partial_i\vy_h|^2\right)^2  \leq \sum_{i,j=1}^2\left(\int_{\K}\partial_i\vy_h\cdot\partial_j\vy_h\right)^2= \left|\int_{\K}\nabla\vy_h^T\nabla\vy_h\right|^2,
\end{align}
whence, taking advantage of the discrete constraint in $\A^k_{h,\veps}$, we have
\begin{equation*}
2^{-\frac12}\|\nabla_h\vy_h\|_{L^2(\Omega)}^2 \leq  \sum_{T\in\Th}\left|\int_{\K}\nabla\vy_h^T\nabla\vy_h\right|
\le \sum_{T\in\Th}\left|\int_{\K}\nabla\vy_h^T\nabla\vy_h-g \, \right|+\sum_{T\in\Th}\left|\int_{\K}g \, \right|\le\veps+\|g\|_{L^1(\Omega)}.
\end{equation*}
This ends the proof.
\end{proof}

In view of Lemma~\ref{lemma:Ah_not_empty} ($\A_{h,\veps}^k$ is non-empty),
the existence of a solution to the minimization problem \eqref{prob:min_Eg_h_intro} follows from standard arguments. This is the object of the next proposition. Note that for this result, $h$ and $\veps$ are fixed.

\begin{prop}[existence of discrete solutions] \label{prop:existence_yh}
	Let $h>0$ and $\veps>0$ be such that $\A_{h,\veps}^k \not = \emptyset$. Then there exists at least one solution to the minimization problem \eqref{prob:min_Eg_h_intro}.
\end{prop}
\begin{proof}
	Let $0\le m:=\inf_{\vy_h\in\A_{h,\veps}^k}E_h(\vy_h) < \infty$ and 
	$\{\vy_h^n\}_{n\ge 1}\subset \A_{h,\veps}^k$ be a minimizing sequence
	\begin{equation} \label{eqn:min_sequ}
	\lim_{n\rightarrow\infty}E_h(\vy_h^n) = m.
	\end{equation}
	Because $E_h(\vy_h^n+\vc)=E_h(\vy_h^n)$ and $D_h(\vy_h^n+\vc)=D_h(\vy_h^n)$ for any constant vector $\vc\in\mathbb{R}^3$, we can assume without loss of generality that $\int_\Omega \vy_h^n = 0$.
        Combining estimate \eqref{eqn:fp-smoothing} of Lemma~\ref{estimate-smoothing-interpolation} (discrete Poincar\'e-Friedrichs inequalities), Lemma~\ref{bound-nabla} (gradient estimate) and Theorem~\ref{coercivity:plates} (coercivity of $E_h$), we deduce that $\| \vy_h^n \|_{L^2(\Omega)} \lesssim 1$.
	Because $[\V_h^k]^3$ is finite dimensional, we have that (up to a subsequence)  $\{ \vy_h^n \}_{n\geq 1}$ converges strongly in $[L^2(\Omega)]^3$ to some $\vy_h^\infty \in [\V_h^k]^3$, and so in any norm. In turn, the continuity of the quadratic energy $E_h$ and the prestrain defect $D_h$ guarantee that  
	$$
	E_h(\vy_h^{\infty})= \lim_{n\rightarrow\infty}E_h(\vy_h^n)=\inf_{\vy_h\in\A_{h,\veps}^k}E_h(\vy_h)
	$$
	and $\vy_h^\infty \in \A_{h,\veps}^k$. This proves that $\vy_h^\infty$ is a solution to the minimization problem \eqref{prob:min_Eg_h_intro}.
	\end{proof}
 
\section{$\Gamma$-convergence of $E_h$ and convergence of global minimizers} \label{sec:Gamma}

The convergence of discrete global minimizers of problem \eqref{prob:min_Eg_h_intro} towards a global minimizer of \eqref{prob:min_Eg} follows from the $\Gamma$-convergence of $E_h$ towards $E$ as $h\rightarrow 0$. 
The latter hinges on the so-called lim-inf and lim-sup properties and this section is devoted to the proof of these two properties. However, we start by stating the convergence of discrete global minimizers $\vy_h$ satisfying $E_h(\vy_h) \leq \Lambda$ for a constant $\Lambda$ independent of $h$; see Theorem \ref{conv_global_min} below. For the sake of brevity, the proof is omitted as it closely follows the one of \cite[Theorem 5.1]{bonito2018}, where the deformations of single layer plates subject to an isometry constraint and Dirichlet boundary conditions are considered.  
We require that the prestrain parameter $\veps$ satisfies
\begin{equation}\label{e:pres_param_cond}
\veps \geq Ch \big(\| g\|_{L^1(\Omega)} + \Lambda \big),
\end{equation}
where $C=C(g)$ is a constant only depending on the hidden constant in \eqref{eqn:fp-smoothing} and that in
$$
| \vw |_{H^2(\Omega)}^2 \lesssim E(\vw)  \qquad \forall\, \vw \in \A.
$$
Note that the condition \eqref{e:pres_param_cond} on $\veps$ differs from the one that appears in \cite[Theorem 5.1]{bonito2018}, which involves the boundary data but not the metric $g$.

\begin{thm}[convergence of global minimizers]\label{conv_global_min}
Let $\{\vy_h\}_{h>0} \subset\A_{h,\veps}^k$ be a sequence of functions such that $E_h(\vy_h)\leq \Lambda$ for a constant $\Lambda$ independent of $h$ and let the prestrain defect parameter $\veps$ satisfy \eqref{e:pres_param_cond}.
If $\vy_h \in \A_{h,\veps}^k$ is an almost global minimizer of $E_h$ in the sense that 
$$
E_h(\vy_h) \leq \inf_{\vw_h \in \A_{h,\veps}^k} E_h(\vw_h) + \sigma,
$$
where $\sigma,\veps \to 0$ as $h\to 0$, then $\{ \wb \vy_h  \}_{h>0}$ with $\wb \vy_h:= \vy_h- \strokedint_{\Omega} \vy_h$ is precompact in $[L^2(\Omega)]^3$  and every cluster point $\wb \vy$ of $\wb \vy_h$ belongs to $\A$
and is a global minimizer of $E$, namely 
$
	E(\bar \vy) = \inf_{\vw \in \A}E(\vw).
$
Moreover, up to a subsequence (not relabeled), the energies converge
$$
	\lim_{h  \rightarrow 0}E_h(\vy_h)=E(\bar \vy).
$$
\end{thm}

We now provide a proof of the lim-inf property.

\begin{thm}[lim-inf of $E_h$]\label{lim-inf}
Let the prestrain defect parameter $\veps=\veps(h)\rightarrow 0$ as $h\rightarrow 0$.
Let $\{\vy_h\}_{h>0} \subset\A_{h,\veps}^k$ be a sequence of functions such that $E_h(\vy_h)\lesssim 1$. 
Then there exists  $\wb \vy \in \A$ with $\strokedint_{\Omega}\wb \vy=0$ such that (up to a subsequence) the shifted sequence $\wb\vy_h:=\vy_h-\strokedint_{\Omega}\vy_h \in\A_{h,\veps}^k$ satisfies $\wb\vy_h\to\wb \vy$, $\nabla_h\wb\vy_h\to\nabla\wb \vy$, $H_h(\wb\vy_h)\rightharpoonup D^2 \wb \vy$ in $L^2(\Omega)$ as $h,\veps\to0$ and
\[
E(\wb \vy)\le\liminf\limits_{h\to0}E_h(\wb\vy_h).
\]
\end{thm}
\begin{proof}
We proceed in several steps.

\smallskip\noindent
{\it Step 1: Accumulation point.}
Lemma~\ref{bound-nabla} (gradient estimate), Theorem~\ref{coercivity:plates} (coercivity of $E_h$) and the uniform bound $E_h(\vy_h) \lesssim 1$ guarantee that
\begin{equation}\label{e:unif_bound}
\| \nabla_h \vy_h \|^2_{L^2(\Omega)} + | \vy_h |^2_{H^2_h(\Omega)} \lesssim \veps + \| g\|_{L^1(\Omega)} + E_h(\vy_h) \lesssim 1.
\end{equation}
We can thus invoke Lemma~\ref{l:compactness} (compactness) with $v_h= y_{h,m}$ and deduce the existence of $\wb \vy \in [H^2(\Omega)]^3$ with $\strokedint_{\Omega}\wb \vy=0$ such that $\wb\vy_h\to\wb \vy$ and $\nabla_h\wb\vy_h\to\nabla\wb \vy$ as $h \to 0$.

\smallskip\noindent
{\it Step 2: Admissible deformation.} 
We now show that $\wb \vy\in \A$. Proceeding as in Step 4 of Proposition 5.1 in  \cite{bonito2018}, which considers the case $g=I$, we have that
\begin{equation}\label{l1cons}
\|\nabla_h\wb\vy_h^T\nabla_h\wb\vy_h-g\|_{L^1(\Omega)}\lesssim h(\| D_h^2 \wb\vy_h\|_{L^2(\Omega)} \| \nabla_h \wb \vy_h \|_{L^2(\Omega)} +\| \nabla g\|_{L^1(\Omega)}) +D_h(\wb\vy_h) \lesssim h+ \veps,
\end{equation}
where we used \eqref{e:unif_bound}, the fact that $D^2_h\wb \vy_h=D^2_h\vy_h$, $\nabla_h\wb \vy_h=\nabla_h\vy_h$ and $D_h(\wb \vy_h) = D_h(\vy_h) \leq \veps$ for $\vy_h \in \A_{h,\veps}^k$.
Hence, taking advantage of the relation~\eqref{e:add_sub_contraint} and the  convergence $\nabla_h\wb\vy_h\to\nabla\wb \vy$ in $[L^2(\Omega)]^3$, we deduce that
\begin{equation*}
\|\nabla\wb \vy^T\nabla\wb \vy-g\|_{L^1(\Omega)}\le\left(\|\nabla\wb \vy\|_{L^2(\Omega)}+\|\nabla_h\wb\vy_h\|_{L^2(\Omega)}\right)\|\nabla_h\wb\vy_h-\nabla\wb \vy\|_{L^2(\Omega)}+\|\nabla_h\wb\vy_h^T\nabla_h\wb\vy_h-g\|_{L^1(\Omega)}\to0
\end{equation*}
as $h,\veps\to0$. This proves that $\nabla\wb \vy^T\nabla\wb \vy=g$ a.e. in $\Omega$, and hence $\wb \vy\in\A$. 

\smallskip\noindent
{\it Step 3: $\lim-\inf$ property.} 
Thanks to Lemma \ref{weak-conv} (weak convergence of $H_h$)  we have $H_h(\wb y_{h,m})\rightharpoonup D^2 \wb y_m$ as $h\rightarrow 0$ for $m=1,2,3$, and so 
\begin{equation}\label{e:weakg}
g^{-\frac{1}{2}}H_h(\wb y_{h,m})g^{-\frac{1}{2}}\rightharpoonup g^{-\frac{1}{2}}D^2 \wb y_m g^{-\frac{1}{2}} \qquad  \textrm{as }h\rightarrow 0, \qquad m=1,2,3.
\end{equation}
Thus, the weak lower semi-continuity of the $L^2(\Omega)$ norm implies that 
\begin{equation*}
\int_{\Omega}|g^{-\frac{1}{2}}D^2 \wb y_mg^{-\frac{1}{2}}|^2\le \liminf_{h\to0}\int_{\Omega}|g^{-\frac{1}{2}}H_h(\wb y_{h,m})g^{-\frac{1}{2}}|^2, \qquad m=1,2,3.
\end{equation*} 

A similar estimate for the trace terms in $E_h$ and $E$ can be derived.
First note that $p:[L^2(\Omega)]^{2\times2}\rightarrow \mathbb{R}$ defined by $p(F):=(\int_{\Omega}\tr(F)^2)^{\frac{1}{2}}$ is convex (it is a semi-norm) and satisfies
\begin{equation}\label{e:p}
p(F_n) \to p(F) \qquad \textrm{when }F_n \to F \quad \textrm{strongly in }[L^2(\Omega)]^{2\times 2}. 
\end{equation}
In particular, $p$ is lower semi-continuous with respect to the strong topology of $[L^2(\Omega)]^{2\times2}$. Consequently it is also lower semi-continuous with respect to the weak topology and so
\begin{equation*}
\int_{\Omega}\tr(g^{-\frac{1}{2}}D^2 \wb y_mg^{-\frac{1}{2}})^2\le \liminf_{h\to0}\int_{\Omega}\tr(g^{-\frac{1}{2}}H_h(\wb y_{h,m})g^{-\frac{1}{2}})^2, \qquad m=1,2,3,
\end{equation*} 
follows from the weak convergence property \eqref{e:weakg}.

It remains to use the fact that the remaining terms in $E_h$ (namely the stabilization terms) are positive, to conclude that $E(\wb \vy)\le \liminf\limits_{h\to0}E_h(\wb\vy_h)$ as desired.    
\end{proof}

We now discuss the lim-sup property. It turns out that in our context, we are able to construct a recovery sequence with strongly converging energies. 

\begin{thm}[lim-sup of $E_h$]\label{lim-sup}
For any $\vy\in \A$, there exists a recovery sequence $\{\vy_h\}_{h>0}\subset\A_{h,\veps}^k\cap [H^1(\Omega)]^3$ such that
$$\vy_h\to\vy \quad \mbox{in } \,\, [L^2(\Omega)]^3 \qquad \mbox{ as } h\to0.$$
Moreover, for $\veps\ge Ch\|y\|_{H^2(\Omega)}^2$, where $C$ is the constant appearing in Lemma~\ref{lemma:Ah_not_empty} ($\A_{h,\veps}^k$ is non-empty), we have $E(\vy) = \lim_{h\to 0} E_h(\vy_h)$.
\end{thm}

\begin{proof}
Consider the sequence $\vy_h:=\mathcal{I}_h^k \vy \in [\V^k_h]^3\cap [H^1(\Omega)]^3$ consisting of the Lagrange interpolants of $\vy$ and note that, thanks to Lemma \ref{lemma:Ah_not_empty}, $\vy_h \in \A_{h,\veps}^k$ since $\veps \geq Ch\| \vy\|_{H^2(\Omega)}^2$ by assumption. Moreover, the fact that $\vy_h \to \vy$ in $[L^2(\Omega)]^3$ as $h\to 0$ stems directly from interpolation estimates.
Furthermore, Lemma \ref{strong-conv} (strong convergence of $H_h$) applied to $v_h=y_{h,m}$, for $m=1,2,3$, yields the convergence in norm
\begin{equation}\label{e:strong_infsup}
\lim_{h\to0}\int_{\Omega}|g^{-\frac{1}{2}}H_h(y_{h,m})g^{-\frac{1}{2}}|^2= \int_{\Omega}|g^{-\frac{1}{2}}D^2y_mg^{-\frac{1}{2}}|^2, \qquad m=1,2,3.
\end{equation}
Similarly for the trace term, thanks to \eqref{e:p} and the strong convergence $g^{-\frac{1}{2}}H_h(y_{h,m})g^{-\frac{1}{2}}\to g^{-\frac{1}{2}}D^2y_mg^{-\frac{1}{2}}$ in $[L^2(\Omega)]^{2\times2}$ for $m=1,2,3$ we have
\begin{equation*}
\lim_{h\to0}\int_{\Omega}\tr(g^{-\frac{1}{2}}H_h(y_{h,m})g^{-\frac{1}{2}})^2 = \int_{\Omega}\tr(g^{-\frac{1}{2}}D^2 y_mg^{-\frac{1}{2}})^2, \qquad m=1,2,3.
\end{equation*} 

Gathering these estimates together with the fact that the stabilization terms in $E_h$ vanish in the limit $h\to 0$, according to Lemma \ref{strong-conv}, we have $E(\vy)=\lim_{h\to0}E_h (\vy_h)$ as desired. 
\end{proof}

\section{Discrete gradient flow}\label{sec:GF}

We advocate a discrete $H^2$-gradient flow to determine local minimizers $\vy_h \in \A_{h,\veps}^k$ of $E_h(\vy_h)$, which is driven by the Hilbert structure induced by the scalar product $(\vv_h,\vw_h)_{H_h^2(\Omega)}$ defined in \eqref{e:H2metric}. We now introduce this flow and discuss its crucial properties.

Given an initial guess  $\vy_h^0\in\A_{h,\veps_{0}}^k$ and a pseudo time-step $\tau>0$, we iteratively compute $\vy_h^{n+1}:=\vy_h^{n}+\delta\vy_h^{n+1}\in [\V_h^k]^3$ that minimizes the functional
\begin{equation}\label{gf:minimization}
\vy_h\mapsto G_h(\vy_h) := \frac{1}{2\tau}\|\vy_h-\vy_h^n\|_{H_h^2(\Omega)}^2+E_h(\vy_h),
\end{equation}
under the \emph{linearized metric constraint} $\delta\vy_h^{n+1}\in\mathcal{F}_h(\vy_h^n)$ for the increment, where
\begin{equation}\label{gf:constraint}
\mathcal{F}_h(\vy_h^n):=\left\{ \vv_h\in [\V_h^k]^3: \,\,
	L_{\K}(\vy_h^n;\vv_h):=\int_{\K}\nabla\vv_h^T\nabla\vy_h^n+(\nabla\vy_h^n)^T\nabla\vv_h=0 \quad \forall\, \K\in\Th\right\}.
\end{equation}
We emphasize that \eqref{gf:minimization} minimizes $E_h(\vy_h)$ but penalizes the deviation of $\vy_h^{n+1}$ from $\vy_h^n$. In the formal limit $\tau\to0$, the first variation of \eqref{gf:minimization} becomes an ODE in $H^2_h(\Omega)$ against the variational derivative $\delta E_h(\vy_h)$ of $E_h(\vy_h)$. The first variation of  \eqref{gf:minimization} does indeed give the first optimality condition:
$\delta\vy_h^{n+1} \in \mathcal{F}_h(\vy_h^n)$ satisfies the Euler-Lagrange system of equations
\begin{equation}\label{gf:variation}
\begin{array}{rcl}
\tau^{-1}(\delta\vy_h^{n+1},\vv_h)_{H_h^2(\Omega)}+a_h(\vy_h^{n+1},\vv_h)= 0 \qquad \forall\,\vv_h \in \mathcal{F}_h(\vy_h^n),
\end{array}
\end{equation}
where $a_h(\vy_h,\vv_h) = \delta E_h(\vy_h)(\vv_h)$ is given by
\begin{equation} \label{def:form_ah}
\begin{split}
a_h(\vy_h,\vv_h)  := & \frac{\mu}{6}\sum_{m=1}^3\int_{\Omega}(\g^{-\frac{1}{2}}H_h(y_{h,m}) \g^{-\frac{1}{2}}):(\g^{-\frac{1}{2}}H_h(v_{h,m}) \g^{-\frac{1}{2}}) \\
 & +\frac{\mu\lambda }{6(2\mu+\lambda)}\sum_{m=1}^3\int_{\Omega}\tr\left(\g^{-\frac{1}{2}}H_h(y_{h,m}) \g^{-\frac{1}{2}}\right)\tr\left(\g^{-\frac{1}{2}}H_h(v_{h,m}) \g^{-\frac{1}{2}}\right) \\
 & +\gamma_1 \big(\h^{-1}\jump{\nabla_h\vy_h},\jump{\nabla_h\vv_h}\big)_{L^2(\Gh^0)}+\gamma_0 \big(\h^{-3}\jump{\vy_h},\jump{\vv_h} \big)_{L^2(\Gh^0)}.
\end{split}
\end{equation}
We refer to \cite{BGNY2020_comp} for an implementation of the method using Lagrange multipliers to enforce the constraint. 
For convenience in the analysis below, we rewrite  \eqref{gf:variation} as
\begin{equation}\label{gf:system}
\tau^{-1}(\delta\vy_h^{n+1},\vv_h)_{H_h^2(\Omega)} +  a_h(\delta\vy_h^{n+1},\vv_h) = -a_h(\vy_h^n,\vv_h) \qquad \forall\, \vv_h \in \mathcal{F}_h(\vy_h^n).
\end{equation}

\begin{remark}[well-posedness]\label{kernel}
Note that the bilinear form $a_h(\cdot,\cdot)$ has a nontrivial kernel in $\mathcal{F}_h(\vy_h^n)$; it contains at least the constants. 
However, $\tau^{-1}(\cdot,\cdot)_{H_h^2(\Omega)} +  a_h(\cdot,\cdot)$ is coercive and continuous on $\mathcal{F}_h(\vy_h^n) \not = \emptyset$ thanks to the $L^2$ term in the $H^2_h$ metric. Hence, the Lax-Milgram theory guarantees the existence and uniqueness of $\delta y_h^{n+1} \in \mathcal{F}_h(\vy_h^n) $ satisfying \eqref{gf:system}.
\end{remark}

We now embark on the study of properties of the discrete gradient flow \eqref{gf:minimization}. We start with a simple observation: since $\vy_h^{n+1}$ minimizes the functional $G_h$ in \eqref{gf:minimization}, we deduce that $G_h(\vy_h^{n+1}) \le G_h(\vy_h^n)$ or equivalently
$
E_h(\vy_h^{n+1}) + \frac{1}{2\tau} \|\delta\vy_h^{n+1}\|_{H_h^2(\Omega)}^2 \le E_h(\vy_h^n).
$
We first show an improved energy reduction at each step that hinges on the quadratic structure of $G_h$.

\begin{prop}[energy decay]\label{prop:energy-decay}
If $\delta\vy_h^{n+1} \in  \mathcal{F}_h(\vy_h^n)$ solves \eqref{gf:system}, then the next iterate $\vy_h^{n+1}=\vy_h^n+\delta\vy_h^{n+1}$ satisfies
\begin{equation}\label{eqn:control_energy}
E_h(\vy_h^{n+1}) + \frac{1}{\tau} \|\delta\vy_h^{n+1}\|_{H_h^2(\Omega)}^2 \le E_h(\vy_h^n).
\end{equation}
\end{prop}
\begin{proof}
It suffices to utilize the identity $a(a-b)=\frac{1}{2}a^2-\frac{1}{2}b^2+\frac{1}{2}(a-b)^2$ to write
\begin{equation} \label{rel_bilinear}
\begin{aligned}  
a_h(\vy_h^{n+1} &,\delta \vy_h^{n+1}) = a_h(\vy_h^{n+1},\vy_h^{n+1}-\vy_h^n) \\
& = \frac 1 2 a_h(\vy_h^{n+1}, \vy_h^{n+1}) - \frac 1 2 a_h(\vy_h^{n}, \vy_h^{n}) + \frac 1 2 a_h(\delta \vy_h^{n+1},\delta  \vy_h^{n+1}) \geq E_h(\vy_h^{n+1}) - E_h(\vy_h^{n}),
\end{aligned}
\end{equation}  
and to replace the left-hand side by $-\frac{1}{\tau} \|\delta\vy_h^{n+1}\|_{H_h^2(\Omega)}^2$ according to \eqref{gf:system} for $\vv_h=\delta\vy_h^{n+1}$.
\end{proof}

Note that \eqref{eqn:control_energy} gives a precise control of the energy decay provided $\delta\vy_h^{n+1}\ne0$. Moreover, upon summing \eqref{eqn:control_energy} over $n=0,1,\ldots,N-1$ for $N\ge 1$, we get the estimate
\begin{equation} \label{e:control_final_energy}
\frac{1}{\tau}\sum_{n=0}^{N-1}\|\delta\vy_h^{n+1}\|_{H_h^2(\Omega)}^2+E_h(\vy_h^N) \leq E_h(\vy_h^0).
\end{equation}

The next result quantifies the prestrain defect of iterates obtained within the gradient flow starting from an initial deformation $\vy_h^0 \in \A_{h,\veps_0}^k$ for some $\veps_0$. 

\begin{prop}[control of metric defect]\label{prop:control-defect}
Let $\vy_h^0\in\A^k_{h,\veps_0}$. Then, for any $N\ge 1$, the $N^{th}$ iterate $\vy_h^N$ of the gradient flow satisfies
\begin{equation} \label{eqn:control_defect}
D_h(\vy_h^N)=\sum_{T\in\mathcal{T}_h}\left|\int_{\K}(\nabla\vy_h^N)^T\nabla\vy_h^N-g \, \right|\le \veps_0+c\tau E_h(\vy_h^0),
\end{equation}
where $c>0$ is the hidden constant in \eqref{eqn:bound_noBC}.
In particular, if $\vy_h := \mathcal I_h^k \vy^0$ is the Lagrange interpolant of some $\vy^0 \in \A$, then
\begin{equation}\label{e:uniform_defect}
D_h(\vy_h^N) \lesssim (h+\tau) \| \vy^0\|_{H^2(\Omega)}^2.
\end{equation}
\end{prop}

\begin{proof}
The argument follows verbatim that of \cite[Lemma 3.4]{bonito2018} and is therefore only sketched. We take advantage of the linearized metric constraint $L_T(\vy_h^{n};\delta\vy_h^{n+1})=0$, encoded in \eqref{gf:constraint} for $\vv_h=\delta\vy_h^{n+1}$, to obtain for all $n\ge 0$
$$
(\nabla_h\vy_h^{n+1})^T\nabla_h\vy_h^{n+1}-g = (\nabla_h\vy_h^{n})^T\nabla_h\vy_h^{n}-g + (\nabla_h \delta \vy_h^{n+1})^T\nabla_h \delta \vy_h^{n+1}.
$$
Therefore, summing for $n=0,...,N-1$ and exploiting telescopic cancellation yield
\[
D_h(\vy_h^N)= \sum_{\K\in\Th}\left|\int_{\K} (\nabla\vy_h^N)^T\nabla\vy_h^N-g\right| \leq  D_h(\vy_h^0) + \sum_{n=0}^{N-1}  \| \nabla_h \delta \vy_h^{n+1} \|_{L^2(\Omega)}^2.
\]
Since $\vy_h^0 \in \A_{h,\veps_0}^k$, using \eqref{eqn:bound_noBC} of Lemma \ref{estimate-smoothing-interpolation} (discrete Poincar\'e-Friedrichs inequalities) implies
\begin{equation*}
D_h(\vy_h^N) \leq 
\veps_0+c\sum_{n=0}^{N-1}\|\delta\vy_h^{n+1}\|_{H_h^2(\Omega)}^2.
\end{equation*}
To derive \eqref{eqn:control_defect}, we employ \eqref{e:control_final_energy} along with $E_h(\cdot) \geq 0$ and realize that
\begin{equation}\label{e:decay_increment}
\sum_{n=0}^{N-1}\|\delta\vy_h^{n+1}\|_{H_h^2(\Omega)}^2 \leq \tau E_h(\vy_h^0).
\end{equation}

On the other hand, if $\vy^0\in\A$, then \eqref{eqn:Eh<=H2} in conjunction with a trace inequality, and the local $H^2$-stability of the Lagrange interpolant imply
\begin{equation} \label{e:bound_E_interp}
E_h(\mathcal I_h^k \vy^0) \lesssim \| \mathcal I_h^k \vy^0 \|_{H^2_h(\Omega)}^2 \lesssim  \| \vy^0 \|_{H^2(\Omega)}^2.
\end{equation}
This, together with Lemma \ref{lemma:Ah_not_empty} ($\A_{h,\veps}^k$ is non-empty), gives the desired estimate \eqref{e:uniform_defect}.
\end{proof}

The control on the prestrain defect offered by Proposition~\ref{prop:control-defect} indicates that $\veps_0$ should tend to zero as $h\to0$. 
We discuss in Section~\ref{sec:preprocessing} the delicate construction of initial deformations $\vy_h^0 \in \mathbb A_{h,\veps_0}^k$, so that $\veps_0 \to 0$ as $h\to0$.

We pointed out in Remark~\ref{kernel} that the bilinear form $a_h(\cdot,\cdot)$ has a nontrivial kernel in $\mathcal{F}_h(\vy_h^n)$ containing the constant vectors and yet the variational problem \eqref{gf:system} uniquely determines the increment $\delta \vy_h^{n+1}$.
  This is reflected in Theorem~\ref{conv_global_min} (convergence of global minimizers) where the sequence $\wb \vy_h:= \vy_h - \strokedint_{\Omega} \vy_h$, rather than $\vy_h$, is precompact. We explore next that the gradient flow preserves deformation averages throughout the evolution.

\begin{prop}[evolution of averages]\label{ave}
Let $\vy_{h}^0 \in [\V_h^k]^3$. Then all the iterates $\vy_h^n$, $n\geq 1$, of the gradient flow \eqref{gf:variation} satisfy 
\begin{equation}\label{e:avg+cons}
\int_{\Omega}\vy_{h }^n=\int_{\Omega}\vy_{h }^0.
\end{equation}
\end{prop}
\begin{proof}
It suffices to choose a constant test function $\vv_h=\mathbf{c}\in\mathcal{F}_h(\vy_h^n)$ in \eqref{gf:variation} to obtain
\begin{equation*}
(\delta\vy_{h }^{n+1},\mathbf{c})_{L^2(\Omega)}= 0,
\end{equation*}
whence \eqref{e:avg+cons} follows immediately.
\end{proof}

In particular, Proposition \ref{ave} implies that $\strokedint_{\Omega}\vy_h^n=\mathbf{0}$ for all the iterates if $\strokedint_{\Omega} \vy_h^0=\mathbf{0}$. The latter can easily be achieved by subtracting $\strokedint_{\Omega} \vy_h^0$ from any initial guess $\vy_h^0$ without affecting $E_h(\vy_h^0)$ or $D_h(\vy_h^0)$. In this case, the sequence $\{\vy_h^N\}_{h>0}$ of outputs of the gradient flow \eqref{gf:variation} satisfies the assumption in Theorem \ref{lim-inf} and is precompact without further shifting.

Energy decreasing gradient flow algorithms are generally not guaranteed to converge to global minimizers.
We address this issue next upon showing that the gradient flow \eqref{gf:variation} reaches a local minimum $\vy_h^\infty$ for $E_h$ in the direction of the tangent plane $\mathcal{F}_h(\vy_h^{\infty})$. This requires, however, the following (possibly degenerate) {\it inf-sup condition}:
for all $n\ge 0$, there exists a constant $\beta_h>0$ independent of $n$ but possibly depending on $h$ such that
\begin{equation} \label{eqn:inf_sup}
\newinf_{\mu_h\in\Lambda_h\vphantom{\vv_h\in [\V_h^k]^3}}\sup_{\vv_h\in [\V_h^k]^3\vphantom{\mu\in\Lambda_h}}\frac{b_h(\vy_h^n;\vv_h,\mu_h)}{\|\vv_h\|_{H_h^2(\Omega)}\|\mu_h\|_{L^2(\Omega)}} \ge \beta_h,
\end{equation}
where $\Lambda_h:=\{ \mu_h \in [\V_h^0]^{2\times 2}: \, \mu_h^T = \mu_h \}$ is the set of Lagrange multipliers and the bilinear form $b_h(\vy_h^n;\cdot,\cdot)$ is defined for any $(\vv_h,\mu_h)\in [\V_h^k]^3\times\Lambda_h$ by
\begin{equation*}
b_h(\vy_h^n;\vv_h,\mu_h):=\sum_{\K\in\Th}\int_{\K}\mu_h:\big(\nabla \vv_h^T\nabla\vy_h^n+(\nabla\vy_h^n)^T\nabla\vv_h\big).
\end{equation*}
The proof of \eqref{eqn:inf_sup} is an open problem, but experiments presented in \cite{BGNY2020_comp} suggest its validity. We notice the mismatch between the function spaces $H_h^2(\Omega)$ and $L^2(\Omega)$, which are natural for $E_h$ but not for $b_h$, and the fact that $\vy_h^n$ is not known to belong to $[W^1_\infty(\Omega)]^3$ uniformly. We stress that an inf-sup condition similar to \eqref{eqn:inf_sup} is valid for an LDG scheme for bilayer plates provided the linearized metric constraint $L_T(\vy_h,\vv_h)=0$ of \eqref{gf:constraint} is enforced pointwise \cite{BNS:21}.

\begin{prop}[limit of gradient flow] \label{prop:local_min}
		Fix $h>0$ and assume that the inf-sup condition \eqref{eqn:inf_sup} holds . Let $\vy_h^0\in\A_{h,\varepsilon_0}$ be such that $E_h(\vy_h^0)<\infty$ and let $\{\vy_h^n\}_{n\ge 1}\subset\A_{h,\veps}^k$ be the sequence produced by the discrete gradient flow \eqref{gf:variation}.		
		Then there exists  $\vy_h^{\infty}\in\A_{h,\veps}^k$ such that (up to a subsequence) $\vy_h^n \rightarrow\vy_h^{\infty}$ as $n\rightarrow\infty$ and $\vy_h^{\infty}$ is a local minimum for $E_h$ in the direction $\mathcal{F}_h(\vy_h^{\infty})$, namely
		\begin{equation} \label{eqn:loc_min_in_F}
		E_h(\vy_h^{\infty}) \leq E_h(\vy_h^{\infty}+\vv_h) \qquad \forall\, \vv_h\in \mathcal{F}_h(\vy_h^{\infty}).
		\end{equation}
\end{prop}

\begin{proof}
Thanks to the energy decay property \eqref{eqn:control_energy} and the average conservation property \eqref{e:avg+cons}, we have that $\sup_{n\ge 1}E_h(\vy_h^n)\le E_h(\vy_h^0)<\infty$ and $\strokedint_{\Omega} \vy_h^n =\strokedint_{\Omega} \vy_h^0$ for all $n\ge 1$. Arguing as in the proof of Proposition \ref{prop:existence_yh} (existence of discrete solutions), we deduce that a subsequence (not relabeled) converges to some $\vy_h^{\infty}\in\A_{h,\veps}^k$ in any norm defined on $[\V_h^k]^3$.

It remains to prove \eqref{eqn:loc_min_in_F}.
Since $E_h$ is quadratic and convex,
we infer that
\begin{equation*}
E_h(\vy_h^{\infty}+\vv_h) \ge E_h(\vy_h^{\infty}) + \delta E_h(\vy_h^{\infty})(\vv_h).
\end{equation*}
Hence, to prove \eqref{eqn:loc_min_in_F}, it suffices to  show that
\begin{equation} \label{eqn:loc_min_in_F_equiv}
a_h(\vy_h^\infty,\vv_h) = \delta E_h(\vy_h^{\infty})(\vv_h) = 0 \qquad \forall\, \vv_h\in\mathcal{F}_h(\vy_h^{\infty}).
\end{equation}
To this end, we take advantage of the inf-sup condition \eqref{eqn:inf_sup} to guarantee the existence of a unique $\lambda_h^{n+1}   \in \Lambda_h$ such that $\delta \vy_h^{n+1}$ in \eqref{gf:variation} satisfies
 \begin{equation}\label{e:saddle}
\tau^{-1}(\delta\vy_h^{n+1},\vv_h)_{H_h^2(\Omega)} +  a_h(\vy_h^{n+1},\vv_h) +b_h(\vy_h^n;\vv_h,\lambda_h^{n+1})  = 0 \qquad \forall\, \vv_h  \in [\V_h^k]^3.
\end{equation}
Now, from the estimate \eqref{e:decay_increment} on the increments $\delta \vy_h^n$, we deduce that
$\lim_{n\to \infty} \delta \vy_h^n = 0$
and so taking the limit as $n\to \infty$ in \eqref{e:saddle} yields
\begin{equation} \label{limit_bhn_step1}
\lim_{n\rightarrow\infty}b_h(\vy_h^n;\vv_h,\lambda_h^{n+1})=-a_h(\vy_h^{\infty},\vv_h) \qquad \forall\, \vv_h\in [\V_h^k]^3.
\end{equation}
This in conjunction with the inf-sup condition \eqref{eqn:inf_sup} yields
\begin{equation*}
\sup_{n\ge0}\|\lambda_h^{n+1}\|_{L^2(\Omega)}\leq \frac{1}{\beta_h}\sup_{n\ge0\vphantom{\vv_h\in [\V_h^k]^3}}\sup_{\vv_h\in [\V_h^k]^3}\frac{b_h(\vy_h^n;\vv_h,\lambda_h^{n+1})}{\|\vv_h\|_{H_h^2(\Omega)}}<\infty.
\end{equation*}
This in turn implies the existence of  $\lambda_h^{\infty}\in\Lambda_h$ such that (up to a subsequence) $\lambda_h^n \to \lambda_h^\infty$ in any norm defined on the finite dimensional space $\Lambda_h$. In particular, \eqref{limit_bhn_step1} becomes
\begin{equation*}
\sum_{\K\in\Th}\lambda_h^{\infty}|_T: L_T(\vy_h^\infty,\vv_h) =   
\sum_{\K\in\Th}\int_{\K}\lambda_h^{\infty}:(\nabla \vv_h^T\nabla\vy_h^{\infty}+(\nabla\vy_h^{\infty})^T\nabla\vv_h) = -a_h(\vy_h^{\infty},\vv_h)
\end{equation*}
for all $\vv_h\in [\V_h^k]^3$. In particular, if $\vv_h\in\mathcal{F}_h(\vy_h^{\infty})$, then $L_T(\vy_h^\infty,\vv_h)=0$ for all $T\in\Th$ according to \eqref{gf:constraint}. This implies \eqref{eqn:loc_min_in_F_equiv} and ends the proof.
\end{proof}

\section{Preprocessing: initial data preparation}\label{sec:preprocessing}

Propositions ~\ref{prop:energy-decay} (energy decay) and~\ref{prop:control-defect} (control of metric defect) guarantee that the gradient flow \eqref{gf:system} constructs iterates $\vy_h^n$ with decreasing energy $E_h(\vy_h^n)$ (as long as the increments do not vanish) and with prestrain defect $D_h(\vy_h^n)$ smaller than $\veps= D_h(\vy_h^0) + C E_h(\vy_h^0) \tau$. Since $\tau=\mathcal{O}(h)$ in practice, we realize that the choice of the initial deformation $\vy_h^0 \in [\V_h^k]^3$ dictates the size of $\veps$, which must satisfy $\veps \to 0$ as $h\to 0$ for Theorem~\ref{conv_global_min} (convergence of global minimizers) to hold.
The assumption $\A \not = \emptyset$ along with Lemma \ref{lemma:Ah_not_empty} ($\A_{h,\veps}^k$ is non-empty) implies that such $\vy_h^0$'s exist.
Yet, their construction is a delicate issue, especially when $g \not = I_2$ as in the present study. This is the objective of this section.

Motivated by the numerical experiments presented in \cite{BGNY2020_comp}, we propose a \textit{metric preprocessing} algorithm consisting of a discrete $H^2$-gradient flow for the preprocessing energy
\begin{equation} \label{def:E_prestrain_PP}
E_h^p(\vy_h):= E^s_h(\vy_h)+\sigma_h E^b_h(\vy_h),
\end{equation}
where
\begin{equation} \label{def:E_prestrain_S}
E^s_h(\vy_h)  :=  \frac{1}{2}\int_{\Omega} \left|\nabla_h\vy_h^T\nabla_h\vy_h-g \right|^2
\end{equation}
is a (simplified) discrete stretching energy and
\begin{equation} \label{def:E_prestrain_B}
E^b_h(\vy_h)  :=  \frac{1}{2} \left(\int_{\Omega}|g^{-\frac{1}{2}}H_h(\vy_h)g^{-\frac{1}{2}}|^2+\|\h^{-\frac12}\jump{\nabla_h\vy_h}\|_{L^2(\Gh^0)}^2+\|\h^{-\frac32}\jump{\vy_h}\|_{L^2(\Gh^0)}^2\right)
\end{equation}
denotes (part of) the discrete bending energy, and $\sigma_h\ge 0$ is a (small) parameter which may depend on $h$.
Note that the stretching energy $E_h^s$ controls the prestrain defect 
\begin{equation}\label{eqn:defect-Es}
D_h(\vy_h)\le\|(\nabla_h\vy_h)^T\nabla_h\vy_h-g\|_{L^1(\Omega)}\lesssim\|(\nabla_h\vy_h)^T\nabla_h\vy_h-g\|_{L^2(\Omega)}\approx E^s_h(\vy_h)^{\frac{1}{2}}
\end{equation}
for all $\vy_h\in[\V_h^k]^3$, while the bending energy $E_h^b$ controls the $H^2_h(\Omega)$ semi-norm in view of Lemma~\ref{coercivity:hessian}.
As we shall see, for a pseudo time-step sufficiently small, the preprocessing gradient flow produces sequences of deformations $\{ \vy_h^n \}_{n\ge 0}$  
so that $\{{E_h^p}(\vy_h^n)\}_{n\ge 0}$ is decreasing as long as the increments do not vanish. Therefore, for any $\sigma_h\ge 0$ we have $D_h(\vy_h^n)\lesssim E_h^p(\vy_h^n)^{\frac12}$ and if $\sigma_h>0$ and the $N$-th iterate $\vy_h^N$ satisfies $E_h^p(\vy_h^N)\lesssim \sigma_h$, then both $D_h(\vy_h^N) \lesssim \sigma_h^{\frac12}$ and $E_h(\vy_h^N) \lesssim E_h^b(\vy_h^N)\lesssim 1$ for the (full) bending energy \eqref{e:Eh_intro}. The total energy $E_h^p$ is inspired by the pre-asymptotic model reduction of \cite{BGNY2020_comp} as well as the methodology of augmented Lagrangian \cite{fortin2000}.
We recall that $\sigma_h$ scales like the square of the (three-dimensional) plate thickness and tends to $0$  \cite[equations (18) and (19)]{BGNY2020_comp}, which motivates the choice $\sigma_h\approx h^2$. 
In practice, however, the contribution of $E_h^b$ to $E_h^p$ makes negligible difference in computations,
and the numerical experiments of \cite{BGNY2020_comp} are done with $\sigma_h = 0$.  

Furthermore, to cope with the non-quadratic nature of the stretching energy $E_h^s$, the gradient flow is linearized at the previous iterate  and reads: Starting from any initial guess $\vy_h^0\in[\V_h^k]^3$, compute recursively $\vy^{n+1}_h:=\vy^n_h+\delta \vy_h^{n+1}$ where $\delta \vy_h^{n+1}\in [\V_h^k]^3$ satisfies
\begin{equation}\label{pre-gf:variation2}
\begin{split}
\tau^{-1}(\delta\vy_h^{n+1},\vv_h)_{H_h^2(\Omega)}&+a_h^s(\vy_h^n;\delta\vy_h^{n+1},\vv_h)+\sigma_h a_h^b(\delta\vy_h^{n+1},\vv_h)\\
&  =
-a_h^s(\vy_h^n; \vy^n_h,\vv_h)- \sigma_h a_h^b(\vy^n_h,\vv_h) \qquad \forall\, \vv_h\in[\V_h^k]^3.
\end{split}\end{equation}
Here, $(\cdot,\cdot)_{H^2_h(\Omega)}$ is defined in \eqref{e:H2metric} and
\begin{align}
a_h^s(\vy_h^n;\vy_h,\vv_h)  &:= \int_{\Omega}\big(\nabla_h\vv_h^T\nabla_h\vy_h+\nabla_h\vy_h^T\nabla_h\vv_h\big):\big((\nabla_h\vy_h^n)^T\nabla_h\vy_h^n-g\big) \label{pre-gf:bilinear} \\
a_h^b(\vy_h,\vv_h) &:= \int_{\Omega}\big(g^{-\frac{1}{2}}H_h(\vy_h)g^{-\frac{1}{2}}\big):\big(g^{-\frac{1}{2}}H_h(\vv_h)g^{-\frac{1}{2}}\big) \nonumber \\
&+ \big(\h^{-1}\jump{\nabla_h\vy_h},\jump{\nabla_h\vv_h}\big)_{L^2(\Gh^0)}+\big(\h^{-3}\jump{\vy_h},\jump{\vv_h}\big)_{L^2(\Gh^0)}, \label{pre-gf:bilinear_B}
\end{align}
and we use $\tau$ to denote a pseudo time-step parameter. Note that to avoid an overload of symbols, we used the same notation as for the pseudo time-step in the main gradient flow. However, these two pseudo time-steps do not need to take the same value.

We start by showing that the variational problem \eqref{pre-gf:variation2} has a unique solution. 
In preparation, we note that the following discrete Sobolev inequality holds
\begin{equation}\label{eqn:discrete-sobolev}
\|v_h\|_{L^4(\Omega)}\lesssim \|\nabla_hv_h\|_{L^2(\Omega)}+\|\h^{-\frac12}\jump{v_h}\|_{L^2(\Gh^0)}+\|v_h\|_{L^2(\Omega)} \qquad \forall\, v_h\in\mathbb{E}(\Th).
\end{equation}
To see this, we resort to the smoothing interpolation operator $\Pi_h$ (see Section \ref{sec:discrete}) and invoke  an inverse inequality and a standard Sobolev inequality for $\Pi_hv_h\in H^1(\Omega)$ to deduce
\begin{equation*}
\|v_h\|_{L^4(\Omega)}\lesssim\|v_h-\Pi_hv_h\|_{L^4(\Omega)}+\|\Pi_hv_h\|_{L^4(\Omega)} \lesssim \|\h^{-1}(v_h-\Pi_hv_h)\|_{L^2(\Omega)}+\|\Pi_hv_h\|_{H^1(\Omega)}.
\end{equation*}
The claim follows from the stability and approximability estimate \eqref{eqn:smooth_interp_H1} of $\Pi_h$. For later use, we record that the discrete Poincar\'e-Friedrichs inequality \eqref{eqn:bound_noBC} and the discrete Sobolev inequality \eqref{eqn:discrete-sobolev} applied to each component of $\nabla_h \vv_h \in [\mathbb{E}(\Th)]^{3\times 2}$ for $\vv_h \in [\V_h^k]^3$ yield
\begin{equation}\label{eqn:discrete-sobolev-grad}
\|\nabla_h \vv_h\|_{L^4(\Omega)}\lesssim \| \vv_h \|_{H^2_h(\Omega)}.
\end{equation}
 
The next proposition concerns the form $a_h^s(\cdot;\cdot,\cdot)$ and is key to guarantee that the hypotheses of the Lax-Milgram Lemma are satisfied.

\begin{lemma}[solvability of \eqref{pre-gf:variation2}]\label{pre-coercivity}
There exists a constant $C_p$ independent of $h$ such that 
\begin{equation} \label{eqn:continuity_sh}
|a_h^s(\vz_h;\vv_h,\vw_h)|\le C_p E^s_h(\vz_h)^{\frac{1}{2}}\|\vv_h\|_{H_h^2(\Omega)}\|\vw_h\|_{H_h^2(\Omega)} \qquad \forall\, \vv_h,\vw_h, \vz_h \in [\V_h^k]^3,
\end{equation}
Moreover, for $\vz_h\in[\V_h^k]^3$ and $\tau$ satisfying 
\begin{equation}\label{e:cond_step_pp}
\tau\leq \big(1+C_p  {E_h^s}(\vz_h)^{\frac{1}{2}}\big)^{-1},
\end{equation}
we have
\begin{equation} \label{eqn:coercivity_sh}
\|\vv_h\|_{H_h^2(\Omega)}^2\le \frac{1}{\tau}(\vv_h,\vv_h)_{H_h^2(\Omega)}+a_h^s(\vz_h;\vv_h,\vv_h) \qquad \forall\, \vv_h \in[\V_h^k]^3.
\end{equation}
Consequently, there exists a unique solution to \eqref{pre-gf:variation2} provided 
$\tau$ satisfies \eqref{e:cond_step_pp} with $\vz_h = \vy_h^n$.
\end{lemma}

\begin{proof}
Let $\vv_h,\vw_h, \vz_h \in[\V_h^k]^3$ and note that
\begin{equation}\label{e:continuity_int}
|a_h^s(\vz_h;\vv_h,\vw_h)|\le 2\sqrt{2} \|\nabla_h \vv_h\|_{L^4(\Omega)}\|\nabla_h\vw_h\|_{L^4(\Omega)} E_h^s(\vz_h)^{\frac12}.
\end{equation} 
Then \eqref{eqn:continuity_sh} follows from the discrete Sobolev inequality \eqref{eqn:discrete-sobolev-grad}. 
The estimate \eqref{eqn:coercivity_sh} results from taking $\vw_h=\vv_h$ in \eqref{eqn:continuity_sh} and the pseudo time-step restriction \eqref{e:cond_step_pp}.
Thanks to \eqref{eqn:continuity_sh}, \eqref{eqn:coercivity_sh} and Lemma~\ref{coercivity:hessian} (discrete $H^2$ semi-norm equivalence), the Lax-Milgram theory applies to guarantee the existence and uniqueness of a solution to \eqref{pre-gf:variation2}.
\end{proof}

The main result of this section is next. 
It shows that the preprocessing energy $E_h^p(\vy_h^{n+1})$ of the deformation $\vy_h^{n+1}$ obtained after one step of the linearized gradient flow \eqref{pre-gf:variation2} is smaller than the energy $E_h^p(\vy_h^{n})$ of the previous iterate provided that the pseudo time-step satisfies $\tau \leq \frac12c_h(\vy_h^n)$, where
\begin{equation}\label{eq:c_n}
c_h(\vy_h^n) := \min\left\{\left(1+C_pE^p_h(\vy_h^{n})^{\frac12}\right)^{-1}, d_h(\vy_h^n)^{-1} \right\}
\end{equation}
with
\begin{equation}\label{eq:d_n}
d_h(\vy_h^n) := \frac{C_p}{2}E^p_h(\vy_h^{n})^{\frac12}+\frac{\wt C_p}{2}\Big(h_{\min}^{-1}\big(E^p_h(\vy_h^{n})+1\big)\big(E^p_h(\vy_h^{n})^{\frac12}+\|g\|_{L^1(\Omega)}\big)+\sigma_hE^p_h(\vy_h^{n})\Big),
\end{equation}
where $h_{\min}:= \min_{T \in \Th} h_T$ and $\wt C_p$ is a constant independent of $n$ and $h$ (to be determined in Proposition~\ref{p:energy_pp}).
While the above restriction on $\tau$ depends on $E^p_h(\vy_h^n)$, we show in the subsequent Corollary~\ref{c:unif_decay} that $c_h(\vy_h^n) \geq c$ for a constant $c$ independent of $h$ and $n$. 

\begin{prop}[energy decay for prestrain preprocessing]\label{energy-decay-preprocessing}\label{p:energy_pp}
Let $\sigma_h\ge 0$. Let $\vy_h^n \in[\V_h^k]^3$ and assume that $\tau \leq \frac 1 2 c_h(\vy_h^n)$ where $c_h(\vy_h^n)$ is defined in \eqref{eq:c_n}.
If $\delta\vy_h^{n+1} \in[\V_h^k]^3$ is the unique solution to \eqref{pre-gf:variation2}, then the new iterate $\vy_h^{n+1}:= \vy_h^n + \delta \vy_h^{n+1}$ satisfies
\begin{equation}\label{eqn:energy-estimate-one}
{E_h^p}(\vy_h^{n+1}) +  \frac 1{2\tau}\|\delta\vy_h^{n+1}\|_{H^2_h(\Omega)}^2 \leq {E_h^p}(\vy_h^n).
\end{equation}
\end{prop}

\begin{proof}
Because $\tau \leq \frac 1 2 c_h(\vy_h^n)$ and $E_h^s(\vy_h^n)\leq E_h^p(\vy_h^n)$, $\tau$ satisfies the assumption \eqref{e:cond_step_pp} of Lemma~\ref{pre-coercivity} (solvability of \eqref{pre-gf:variation2}) and thus there exists a unique solution $\delta\vy_h^{n+1}\in[\V^k_h]^3$ to \eqref{pre-gf:variation2}.
Next we take $\vv_h = \delta \vy_h^{n+1}$ in \eqref{pre-gf:variation2} to obtain 
\begin{equation}\label{energy-eqn-0}
\tau^{-1}\|\delta\vy_h^{n+1}\|^2_{H_h^2(\Omega)}+a^s_h(\vy_h^n;\vy_h^{n+1},\delta\vy_h^{n+1})+\sigma_ha^b_h(\vy_h^{n+1},\delta\vy_h^{n+1})=0
\end{equation}
 and proceed in several steps. In contrast to Proposition \ref{prop:energy-decay} (energy decay), the main difficulty is that $a^s_h$ is quadratic in its first argument.\\
\smallskip\noindent
{\it Step 1: Energy relation.}
Since $a^s_h(\vy_h^n;\cdot,\cdot)$ is bilinear and symmetric,
arguing as in \eqref{rel_bilinear} yields
\begin{equation}\label{bilinear-identity}
a^s_h(\vy_h^n;\vy_h^{n+1},\delta\vy_h^{n+1})=\frac{1}{2}a^s_h(\vy_h^n;\vy_h^{n+1},\vy_h^{n+1})-\frac{1}{2}a^s_h(\vy_h^n;\vy_h^n,\vy_h^n)+\frac{1}{2}a^s_h(\vy_h^n;\delta\vy_h^{n+1},\delta\vy_h^{n+1}).
\end{equation}
Furthermore, using the identity $(a-b)b=\frac{1}{2}a^2-\frac{1}{2}b^2-\frac{1}{2}(a-b)^2$, we have
\begin{align*}
\frac{1}{2}a^s_h(\vy_h^n;\vy_h^{n+1},\vy_h^{n+1})-\frac{1}{2}a^s_h(\vy_h^n;\vy_h^n,\vy_h^n)
& = \int_{\Omega} W_h^n: \big((\nabla_h\vy_h^{n})^T\nabla_h\vy_h^{n}-g \big) \\
& = E_h^s(\vy_h^{n+1}) - E_h^s(\vy_h^n) - \frac 1 2 \| W_h^n \|_{L^2(\Omega)}^2,
\end{align*}
where 
\begin{align}\label{e:Wh}
W_h^n&:=(\nabla_h\vy_h^{n+1})^T\nabla_h\vy_h^{n+1}-(\nabla_h\vy_h^{n})^T\nabla_h\vy_h^{n}.
\end{align}
Therefore, we are able to express $a^s_h(\vy_h^n;\vy_h^{n+1},\delta \vy_h^{n+1})$ in terms of energies as
\begin{equation*}
a^s_h(\vy_h^n;\vy_h^{n+1},\delta \vy_h^{n+1})=E^s_h(\vy_h^{n+1})-E^s_h(\vy_h^{n})+\frac 1 2a^s_h(\vy_h^n;\delta \vy_h^{n+1},\delta \vy_h^{n+1}) -\frac 1 2 \|W_h^n\|_{L^2(\Omega)}^2.
\end{equation*}
Similarly, noting that $E_h^b(\vv_h)=\frac12a^b_h(\vv_h,\vv_h)$ for any $\vv_h\in[\V_h^k]^3$ is quadratic, we obtain
$$
a^b_h(\vy_h^{n+1},\delta\vy_h^{n+1})=E_h^b(\vy_h^{n+1}) - E_h^b(\vy_h^{n}) + \frac 1 2 a_h^b(\delta \vy_h^{n+1},\delta \vy_h^{n+1})\ge E_h^b(\vy_h^{n+1}) - E_h^b(\vy_h^{n}).
$$
Using these two relations in \eqref{energy-eqn-0}, we arrive at
\begin{equation}\label{energy-decay-exact}
  {E_h^p}(\vy_h^{n+1}) - E^p_h(\vy_h^{n}) + \tau^{-1}\|\delta\vy_h^{n+1}\|^2_{H_h^2(\Omega)}
  \le R^n_h,
\end{equation}
where
\begin{equation}\label{e:Rh}
R^n_h:=\frac{1}{2}\|W_h^n\|_{L^2(\Omega)}^2-\frac{1}{2}a^s_h(\vy_h^n;\delta\vy_h^{n+1},\delta\vy_h^{n+1}).
\end{equation}
\smallskip\noindent
{\it Step 2: Bounds for $R^n_h$.}
We now prove the estimate
\begin{equation}\label{eqn:Rnh}
|R^n_h|\le d_h(\vy_h^n) \|\delta\vy_h^{n+1}\|^2_{H_h^2(\Omega)}
\end{equation}
with $d_h(\vy_h^n)$ defined in \eqref{eq:d_n}. We first apply the continuity property \eqref{eqn:continuity_sh} of $a_h^s$ to get
\begin{equation}\label{estimate-a}
|a^s_h(\vy_h^n;\delta\vy_h^{n+1},\delta\vy_h^{n+1})|\le C_p {E_h^p}(\vy_h^{n})^{\frac12}\|\delta\vy_h^{n+1}\|_{H_h^2(\Omega)}^2.
\end{equation}
Then we note that $W_h^n$ can be equivalently written as
$$
W_h^n = (\nabla_h\delta\vy_h^{n+1})^T\nabla_h\vy_h^{n}+(\nabla_h\vy_h^{n})^T\nabla_h\delta\vy_h^{n+1}+(\nabla_h\delta\vy_h^{n+1})^T\nabla_h\delta\vy_h^{n+1}
$$
whence, resorting to
the discrete Sobolev inequality \eqref{eqn:discrete-sobolev-grad}, we obtain
\begin{equation}\label{bound1_for_Whn}
\|W_h^n\|_{L^2(\Omega)}^2 \lesssim \big(\|\nabla_h\vy_h^n\|_{L^4(\Omega)}^2 + \|\delta\vy_h^{n+1}\|_{H_h^2(\Omega)}^2 \big) \|\delta\vy_h^{n+1}\|_{H_h^2(\Omega)}^2.
\end{equation}
To derive \eqref{eqn:Rnh} we estimate $\|\delta\vy_h^{n+1}\|_{H_h^2(\Omega)}^2$ in terms of $\|\nabla_h\vy_h^n\|_{L^4(\Omega)}^2$. To this end, we note that $\tau \leq \frac12c_h(\vy_h^n) < (1+C_p E_h^s(\vy_h^n)^{\frac12})^{-1}$ and apply the coercivity estimate \eqref{eqn:coercivity_sh} together with the positivity of $a_h^b(\cdot,\cdot)$ and the gradient flow equation \eqref{pre-gf:variation2} satisfied by $\delta \vy_h^{n+1}$ to derive
\begin{align*}
\|\delta\vy_h^{n+1}\|^2_{H_h^2(\Omega)}&\le\tau^{-1}\|\delta\vy_h^{n+1}\|^2_{H_h^2(\Omega)}+a^s_h(\vy_h^n;\delta\vy_h^{n+1},\delta\vy_h^{n+1})+\sigma_ha^b_h(\delta\vy_h^{n+1},\delta\vy_h^{n+1}) \\
&=-a^s_h(\vy_h^n;\vy_h^n,\delta\vy_h^{n+1})-\sigma_ha^b_h(\vy_h^n,\delta\vy_h^{n+1}).
\end{align*}
The $H^2$ semi-norm equivalence estimates \eqref{eqn:Eh<=H2} and \eqref{eqn:ineq_h2} show that $E_h^b(\cdot) \sim a_h^b(\cdot,\cdot) \sim | \cdot |^2_{H^2_h(\Omega)}$ on $[\V_h^k]^3$.
Hence, from the  continuity property \eqref{e:continuity_int} of $a_h^s$ and \eqref{eqn:discrete-sobolev-grad}, we infer that 
  \begin{equation*}
\|\delta\vy_h^{n+1}\|_{H_h^2(\Omega)} \lesssim E^p_h(\vy_h^{n})^{\frac12}\big(\|\nabla_h\vy_h^n\|_{L^4(\Omega)}+\sigma_h^{\frac12}\big).
\end{equation*}
Inserting this into \eqref{bound1_for_Whn} gives
\[
\|W_h^n\|_{L^2(\Omega)}^2\lesssim \Big( \big( 1 + E_h^p(\vy_h^n) \big) \|\nabla_h\vy^n_h\|_{L^4(\Omega)}^2 + \sigma_h E_h^p(\vy_h^n) \Big)
\|\delta\vy_h^{n+1}\|^2_{H_h^2(\Omega)}.
\]
We tackle $\|\nabla_h\vy^n_h\|_{L^4(\Omega)}$ via the inverse inequality $\|\nabla_h\vy^n_h\|_{L^4(\Omega)}^2 \lesssim h_{\min}^{-1}\|\nabla_h\vy^n_h\|_{L^2(\Omega)}^2$ and
\begin{equation*}
	\|\nabla_h\vy_h^n\|_{L^2(\Omega)}^2\lesssim E^s_h(\vy_h^n)^{\frac{1}{2}}+\|g\|_{L^1(\Omega)}\lesssim E^p_h(\vy_h^n)^{\frac{1}{2}}+\|g\|_{L^1(\Omega)},
	\end{equation*}
a relation directly following from \eqref{eqn:transpose-estimate}. Altogether, we get
\begin{equation*}
\|W_h^n\|_{L^2(\Omega)}^2\le \wt C_ p\Big(h_{\min}^{-1}\big(E^p_h(\vy_h^{n})+1\big)\big(E^p_h(\vy_h^{n})^{\frac12}+\|g\|_{L^1(\Omega)}\big)+ \sigma_h E^p_h(\vy_h^{n})\Big)\|\delta\vy_h^{n+1}\|^2_{H_h^2(\Omega)}
\end{equation*}
for some constant $\wt C_p$ independent of $n$ and $h$. 
This together with \eqref{estimate-a} yields \eqref{eqn:Rnh}.

\medskip\noindent
{\it Step 3: Conditional energy decay.}
Substituting \eqref{eqn:Rnh} into  \eqref{energy-decay-exact} we observe that
\begin{equation}\label{energy-decay-2}
{E_h^p}(\vy_h^{n+1}) + \big(\tau^{-1}-d_h(\vy_h^n)\big)\|\delta\vy_h^{n+1}\|^2_{H_h^2(\Omega)} \le E^p_h(\vy_h^{n}).
\end{equation}
The desired energy decay is obtained upon realizing that $d_h(\vy_h^n)\leq c_h(\vy_h^n)^{-1}\le(2\tau)^{-1}$, which follows from the definition \eqref{eq:c_n} of $c_h(\vy_h^n)$ and the assumption on $\tau$.
\end{proof}

\begin{cor}[uniform energy decay]\label{c:unif_decay}
Let $\vy_h^0\in[\V_h^k]^3$ and assume that $\tau \leq \frac 1 2 c_h(\vy_h^0)$. 
The sequence of successive iterates $\vy_h^{n+1}:=\vy_h^n+\delta\vy_h^{n+1}$, $n\geq 0$, where $\delta\vy_h^{n+1}\in[\V_h^k]^3$ satisfies \eqref{pre-gf:variation2}, is well defined. Furthermore, the sequence $\{ c_h(\vy_h^n) \}_{n\ge 0}$ is nondecreasing and there holds
\begin{equation} \label{eqn:decay_E_tilde}
{E_h^p}(\vy_h^{N+1})+\frac{1}{2\tau}\sum_{n=0}^N\|\delta\vy_h^{n+1}\|_{H^2_h(\Omega)}^2\le E_h^p(\vy_h^{0}) \qquad\forall \, N\ge 0.
\end{equation} 
\end{cor}
\begin{proof}
We proceed by induction to prove that for every $n\ge 0$, $\delta \vy_h^{n+1}$ is well defined and $c_h(\vy_h^{n+1}) \geq c_h(\vy_h^n)$.
We start with $n=0$. In that case, by assumption $\tau \leq\frac12 c_h(\vy_h^{0})$ and Lemma~\ref{pre-coercivity} and Proposition~\ref{p:energy_pp} guarantee that $\delta \vy_h^1\in[\V_h^k]^3$ is  well defined and
$$
E_h^p(\vy_h^1) \leq E_h^p(\vy_h^0).
$$
From the expression \eqref{eq:c_n} of $c_h(\vy_h^n)$, which increases as $E_h^p(\vy_h^n)$ decreases, we also deduce that $c_h(\vy_h^1) \geq c_h(\vy_h^0)$.
For the induction step, we assume that $\{\delta \vy_h^{j}\}_{j=1}^{n}$ is well defined and $c_h(\vy_h^{j}) \geq c_h(\vy_h^{j-1})$, $j=1,..,n$,
which implies $\tau \leq \frac 1 2 c_h(\vy_h^{0}) \leq \frac 1 2 c_h(\vy_h^{n})$.
Therefore, Lemma~\ref{pre-coercivity} (solvability of \eqref{pre-gf:variation2}) and Proposition~\ref{p:energy_pp} (energy decay for prestrain processing) again guarantee that  $\delta \vy_h^{n+1} \in [\V_h^k]^3$ is  well defined and
$$
E_h^p(\vy_h^{n+1}) \leq E_h^p(\vy_h^{n}) \qquad \Longrightarrow  \qquad c_h(\vy_h^{n+1}) \geq c_h(\vy_h^n).
$$
This is the desired property for $n+1$ and concludes the induction argument.

Finally, since the condition $\tau \leq \frac 1 2 c_h(\vy_h^{0}) \leq \frac 1 2 c_h(\vy_h^n)$ holds for all $n=1,\ldots,N$, \eqref{eqn:energy-estimate-one} is valid, whence
summing \eqref{eqn:energy-estimate-one} over $n$ yields~\eqref{eqn:decay_E_tilde}.
\end{proof}

We finish this section by relating the initial deformation $\vy_h^0$ for the main gradient flow \eqref{gf:variation} with the output of the preprocessing gradient flow \eqref{pre-gf:variation2}.

\begin{remark}[choice of $\sigma_h$] \label{rem:Eh_unif}
Under the assumptions of  Corollary~\ref{c:unif_decay} (uniform energy decay), the preprocessing gradient flow produces a sequence of deformations $\{\wt \vy_h^n\}_{n\ge 0}$ with decreasing preprocessing energy $E^p_h(\wt \vy_h^n)$. We assume that the $n_h$-th iterate of the preprocessing gradient flow, denoted $\wt \vy_h^{n_h}$, is such that
$$
E_h^p(\wt \vy_h^{n_h}) \lesssim \sigma_h.
$$
Since $\sigma_h$ scales like the square of the (three-dimensional) plate thickness, according to \eqref{def:E_prestrain_PP} and the pre-asymptotic analysis of \cite{BGNY2020_comp}, a natural choice is $\sigma_h\approx h^2$.
Regardless of this scaling,
in view of \eqref{eqn:defect-Es}, the prestrain defect of $\wt\vy_h^{n_h}$ satisfies
$$
D_h(\wt \vy_h^{n_h}) \lesssim E_h^s(\wt\vy_h^{n_h})^{\frac12} \lesssim \sigma_h^{\frac12}, \qquad \textrm{i.e.,} \qquad \wt \vy_h^{n_h} \in \A_{h,c\sigma_h^{\frac12}}^k,
$$
for a suitable constant $c>0$. Moreover, we  have
$$
E_h^b(\wt \vy_h^{n_h}) \lesssim \sigma_h^{-1}E_h^p(\wt \vy_h^{n_h}) \lesssim 1.
$$ 
This implies that $E_h(\wt \vy_h^{n_h})$ is also uniformly bounded, due to the continuity of $E_h$ and the coercivity of $E_h^b$. As a consequence, the main gradient flow \eqref{gf:variation} with initial deformation $\vy_h^0 = \wt \vy_h^{n_h}$ produces iterates $\vy_h^n$ satisfying $D_h(\vy_h^n)\lesssim \sigma_h^{\frac12}+\tau E_h(\vy_h^0)$ thanks to Proposition~\ref{prop:control-defect} (control of metric defect) and $E_h(\vy_h^n)\lesssim 1$ thanks to \eqref{e:control_final_energy}. In particular, if \eqref{gf:variation} leads to an almost global minimizer $\vy_h^{N_h}$ of the energy $E_h$, then the sequence $\{ \vy_h^{N_h}\}_{h>0}$ satisfies the uniform boundedness assumption of Theorem~\ref{conv_global_min} (convergence of global minimizers).
\end{remark}

\appendix
\section{Equivalence between energies \eqref{E:reduced-bending} and \eqref{def:Eg_D2y}}\label{a:alternate}

The following proposition, first shown in \cite[Proposition 1]{BGNY2020_comp}, justifies the replacement of the highly nonlinear reduced bending energy \eqref{E:reduced-bending} involving the second fundamental form $\II[\vy]$ of the deformation $\vy$ by the quadratic energy \eqref{def:Eg_D2y} involving the Hessian $D^2\vy$. This is critical for the design of the numerical scheme. We sketch the proof for completeness.

\begin{prop}[equivalence of \eqref{E:reduced-bending} and \eqref{def:Eg_D2y}] \label{prop:alt_energy}
If $g\in[H^1(\Omega)\cap L^{\infty}(\Omega)]^{2\times 2}$ is SPD a.e. in $\Omega$ and $\vy=(y_m)_{m=1}^3\in[H^2(\Omega)]^3$ satisfies $\nabla\vy^T\nabla\vy=g$ a.e. in $\Omega$, then there exist two non-negative functions $f_1,f_2 \in L^2(\Omega)$ depending only on $g$ and its partial derivatives such that 
\begin{equation} \label{eqn:link_part1}
\big|\g^{-\frac{1}{2}} \, \II[\vy] \, \g^{-\frac{1}{2}} \big|^2 = \sum_{m=1}^3\big|\g^{-\frac{1}{2}} \, D^2 y_m \, \g^{-\frac{1}{2}} \big|^2 + f_1 \qquad \mbox{a.e. in } \Omega,
\end{equation}
and
\begin{equation} \label{eqn:link_part2}
\tr \big(\g^{-\frac{1}{2}} \, \II[\vy] \, \g^{-\frac{1}{2}} \big)^2 = \sum_{m=1}^3\big|\tr \big(\g^{-\frac{1}{2}} \, D^2y_m \, \g^{-\frac{1}{2}} \big)\big|^2 + f_2 \qquad \mbox{a.e. in } \Omega.
\end{equation}	
\end{prop}
\begin{proof}
Since $g\in[H^1(\Omega)\cap L^{\infty}(\Omega)]^{2\times 2}$ and $\vy\in[H^2(\Omega)]^3$ with $\nabla\vy^T\nabla\vy=g$, we have
$$
|\nabla\vy|^2=\tr(g)\in L^{\infty}(\Omega) \qquad \textrm{and} \qquad \vnu=\frac{\partial_1\vy\times\partial_2\vy}{|\partial_1\vy\times\partial_2\vy|}=\frac{\partial_1\vy\times\partial_2\vy}{\sqrt{\det(g)}}.
$$
As a consequence, we deduce the regularity properties $\nabla\vy\in [H^1(\Omega)\cap L^{\infty}(\Omega)]^{3\times 2}$, $\vnu\in [H^1(\Omega)\cap L^{\infty}(\Omega)]^3$, and $\II[\vy]\in L^2(\Omega)^{2\times 2}$ . In addition, for $i,j\in\{1,2\}$, we represent $\partial_{ij}\vy$ in terms of the basis $\{\partial_1\vy,\partial_2\vy,\vnu\}$ of $\mathbb{R}^3$ and the Christoffel symbols $\Gamma_{ij}^l$ of the surface $\vy(\Omega)$ as follows:
\begin{equation*}
\partial_{ij}\vy = \sum_{l=1}^2\Gamma_{ij}^l \, \partial_l\vy+\II_{ij}[\vy] \, \vnu \qquad \mbox{a.e. in } \Omega;
\end{equation*}
we recall that the symbols $\Gamma_{ij}^l$ are intrinsic quantities that depend only on $g$ and its derivatives but not on $\vy$.
To prove \eqref{eqn:link_part1}, write $a=g^{\frac12}$ and note the validity of the expression
\begin{equation} \label{eqn:l2_decompo}
\Big(\big(a D^2y_k a \big)_{ij}\Big)_{k=1}^3 
= \big( a \II[\vy] a \big)_{ij}\vnu
+ \sum_{m,n=1}^2 a_{im} \Big( \sum_{l=1}^2 \Gamma_{mn}^l \partial_l \vy\Big) a_{nj}.
\end{equation}
Since $\vnu$ is orthogonal to $\{\partial_1\vy,\partial_2\vy\}$ and $\vnu^T\vnu=1$, taking the square of the $l^2$-norm on both sides of \eqref{eqn:l2_decompo} yields
\[
\sum_{k=1}^3 \big( a D^2 y_k a  \big)_{ij}^2 =
\big( a \II[\vy] a  \big)_{ij}^2 + f_{ij},
\]
where $f_{ij}$ does not depend explicitly on $\vy$ but only on $g$ and its derivatives. 
This concludes the proof of \eqref{eqn:link_part1}.
The proof of \eqref{eqn:link_part2} is similar.
\end{proof}

\section{Proofs of Lemma \ref{weak-conv} and Lemma \ref{strong-conv}} \label{sec:weak_strong}

We follow \cite{bonito2018}. We stress that the mesh assumptions of Lemma \ref{weak-conv} (weak convergence of $H_h$) are less restrictive and its proof is simpler than \cite[Proposition 4.3]{bonito2018} due to the simpler structure of the lifting operators $R_h$ and $B_h$ of \eqref{E:global-lifting}.

\begin{proof}[Proof of Lemma \ref{weak-conv}]
Let $\phi\in[C_0^{\infty}(\Omega)]^{2\times2}$. We integrate by parts twice to write
\begin{equation*}
\int_{\Omega}H_h(v_h):\phi = \int_{\Omega}D^2_h v_h:\phi-R_h(\jump{\nabla_h v_h}):\phi + B_h(\jump{v_h}):\phi = T_1+T_2+T_3+T_4+T_5,
\end{equation*}
with
\begin{gather*}
T_1 :=\int_{\Omega}v_h \di(\di\phi), \\
T_2 := -\int_{\Omega} R_h(\jump{\nabla_hv_h}):(\phi-\mathcal{I}_h\phi), \quad
T_3 := \int_{\Omega}B_h(\jump{v_h}):(\phi-\mathcal{I}_h\phi ), \\
T_4 := \sum_{e\in\Eh^0}\int_e \jump{\nabla_hv_h} \cdot \avrg{\phi-\mathcal{I}_h\phi}\vn_e,
\quad
T_5 := -\sum_{e\in\Eh^0}\int_e\jump{v_h}\avrg{\di(\phi-\mathcal{I}_h\phi)}\cdot\vn_e.
\end{gather*}
	Here, $\mathcal{I}_h\phi\in[\V^l_h\cap H^1_0(\Omega)]^{2\times2}$ denotes the Lagrange interpolant of $\phi$ of degree $\min\{l_1,l_2\}$, where $l_1$ and $l_2$ are the polynomial degrees of $R_h$ and $B_h$. 	
	We treat each term $T_i$ separately. By assumption $v_h\to v\in H^2(\Omega)$ in $[L^2(\Omega)]^3$ as $h\rightarrow 0$, whence we have
	\begin{equation*}
	T_1\rightarrow \int_{\Omega}v \di(\di\phi)=-\int_{\Omega}\nabla v \cdot \di\phi=\int_{\Omega}D^2v:\phi \quad \mbox{as } h\rightarrow 0.
	\end{equation*} 
	For $T_2$, we invoke  the assumed uniform boundedness  $|v_h|_{H_h^2(\Omega)} \leq C$ and Lemma \ref{L2bound-lifting} (stability of lifting operators) to get 
$$
	|T_2| \lesssim \|\h^{-\frac{1}{2}}\jump{\nabla_hv_h}\|_{L^2(\Gh^0)}\|\phi-\mathcal{I}_h\phi\|_{L^2(\Omega)} \le C\|\mathcal{I}_h\phi-\phi\|_{L^2(\Omega)} \rightarrow 0 \qquad \textrm{as }h\rightarrow 0.
$$
	Similarly, we have $T_3 \to 0$ as $h\rightarrow 0$.
	To estimate $T_4$, we start with a scaled trace inequality 
	\begin{equation}\label{traceineq}
	\|\mathcal{I}_h\phi-\phi\|_{L^2(e)}\lesssim \|\h^{-\frac{1}{2}}(\mathcal{I}_h\phi-\phi)\|_{L^2(\omega(e))}+\|\h^{\frac{1}{2}}\nabla(\mathcal{I}_h\phi-\phi)\|_{L^2(\omega_e)},
	\end{equation}
	and recall that $\omega_e$ is the union of the two elements adjacent to $e\in\Eh^0$ and that the shape regularity property guarantees that $h_e \approx h_{\K}$ for $\K \subset \omega_e$. This, together with the assumption $|v_h|_{H_h^2(\Omega)} \leq C$ and the shape regularity of $\{ \Th \}_{h>0}$, yields
$$
	|T_4| \lesssim  \Big(\sum_{e\in\Eh^0}\|\h^{-\frac12}\jump{\nabla_hv_h}\|_{L^2(e)}^2\Big)^{\frac{1}{2}}\left( \|\mathcal{I}_h\phi-\phi\|_{L^2(\Omega)}+\|\h \nabla(\mathcal{I}_h\phi-\phi)\|_{L^2(\Omega)}\right)\rightarrow 0 \qquad \textrm{as }h\rightarrow 0.
$$
	Similarly, $T_5 \to 0$ as $h\to 0$. 
	Gathering the above relations for $T_1,...,T_5$, we obtain $\int_{\Omega}H_h(v_h):\phi\to\int_{\Omega} D^2v:\phi$ as $h \to 0$, which is the desired weak convergence property. 
\end{proof}

\begin{proof}[Proof of Lemma \ref{strong-conv}]
  We split the proof into three steps.

  \medskip\noindent
  {\it Step 1: strong convergence of the broken Hessian.}
	We recall that for $k' \geq 1$, the Lagrange interpolation operator $\mathcal I_h^{k'} v$ is locally $H^2$ stable
	\begin{equation} \label{eqn:H2stab}
	\|D^2 \mathcal I_h^{k'} w \|_{L^2(T)}\lesssim |w|_{H^2(T)} \qquad  \forall\, w \in H^2(T), \quad \forall\, T\in\Th, 
	\end{equation}
	and satisfies the following approximation estimates for $0\leq m \leq k'+1$
	\begin{equation}\label{interpolation-estimate}
	\|w-\mathcal I_h^{k'} w\|_{H^m(T)}\lesssim h_T^{k'+1-m} |w|_{H^{k'+1}(T)} \qquad \forall\, w \in H^{k'+1}(T).
	\end{equation}
	These estimates are less standard and somewhat more intricate for subdivisions made of quadrilaterals; we refer to Section 9 of \cite{bonito2018} for their proofs.
	
	We now argue by density. Let $v^{\epsilon}\in C^{\infty}(\Omega)$ be a smooth mollifier of $v$ such that $v^{\epsilon}\rightarrow v$ in $H^2(\Omega)$ as $\epsilon\rightarrow 0$. We also set $v_h^\epsilon:= \mathcal I_h^k v^\epsilon$ and write $v_h-v = v_h - v_h^\epsilon + v_h^\epsilon - v^\epsilon + v^\epsilon - v$ so that employing \eqref{eqn:H2stab}, \eqref{interpolation-estimate} and summing over $T\in \Th$ yield
	\begin{align*}
	\|D_h^2v_h- D^2 v\|_{L^2(\Omega)} \leq C \left( |v-v^{\epsilon}|_{H^2(\Omega)}+ h|v^{\epsilon}|_{H^3(\Omega)}\right)
	\end{align*} 
	for a constant $C$ independent of $h$ and $\epsilon$ because $k \ge 2$. Therefore, for every $\eta>0$, we choose $\epsilon$ sufficiently small so that $C |v-v^{\epsilon}|_{H^2(\Omega)} \leq \eta /2$ and then $h$ sufficiently small so that $ C h|v^{\epsilon}|_{H^3(\Omega)} \leq \eta /2$ to arrive at 
	\begin{align*}
	\|D_h^2v_h- D^2 v \, \|_{L^2(\Omega)} \leq \eta.
	\end{align*} 
	This shows the strong convergence of $D_h^2 v_h$ towards $D^2 v$ in $[L^2(\Omega)]^{2\times 2}$ as $h\rightarrow 0$.

\medskip\noindent
{\it Step 2: strong convergence of lifting operators.}        
We now prove that $R_h(\jump{\nabla_hv_h} )\to 0$ but omit dealing with $B_h(\jump{v_h})\to0$, whose proof follows the same idea. Lemma \ref{L2bound-lifting} (stability of lifting operators) implies
\begin{equation*}
\|R_h(\jump{\nabla_hv_h})\|_{L^2(\Omega)}\lesssim \|\h^{-\frac12}\jump{\nabla_h v_h}\|_{L^2(\Gh^0)} = \|\h^{-\frac12}\jump{\nabla_h(v_h-v)}\|_{L^2(\Gh^0)}
\end{equation*}
because $\jump{\nabla v}|_e=0$ for $e\in\Eh^0$.
Thanks to the scaled trace inequality \eqref{traceineq}, the property $\mathcal{I}^k_h(v_h - v)=0$, and the interpolation estimates \eqref{interpolation-estimate}, we obtain for any $e\in \Eh^0$
\begin{align*}
\|\h^{-\frac12}\jump{\nabla_h(v_h-v)}\|^2_{L^2(e)} &\lesssim h_e^{-2}\|\nabla_h(v_h - v)\|_{L^2(\omega_e)}^2+\|D^2_h(v_h - v)\|_{L^2(\omega_e)}^2\\
&\lesssim \|\h^{-2} \nabla_h(v_h - v-\mathcal{I}^k_h(v_h - v))\|_{L^2(\omega_e)}^2+\|D^2_h(v_h - v)\|_{L^2(\omega_e)}^2\\
&\lesssim \|D_h^2 v_h - D^2 v \|_{L^2(\omega_e)}^2.
\end{align*}
Summing over all interior edges and using the shape-regularity of $\{ \Th \}_{h>0}$, we find that
\begin{equation*}
\|R_h(\jump{\nabla_hv_h})\|_{L^2(\Omega)}\lesssim \|D_h^2 v_h - D^2 v \, \|_{L^2(\Omega)} \to0 \qquad\mbox{as } h \rightarrow 0,
\end{equation*}
which is the desired property.

\medskip\noindent
{\it Step 3: strong convergence of discrete Hessian.}        
  The strong convergence \eqref{eqn:H_strong} of the reconstructed Hessian $H_h(v_h)$ to $D^2v$ follows from the definition \eqref{def:discrHess} of $H_h(v_h)$ and the strong convergence of $D^2_h v_h$, $R_h(\jump{\nabla_hv_h} )$, and $B_h(\jump{v_h} )$ established in Steps 1 and 2.
  \end{proof}

\section{Dirichlet boundary conditions and forcing term} \label{sec:Dirichlet_BC}

We have considered so far \emph{free boundary} conditions. In this case, a physically necessary assumption is that any external forcing $\vf\in[L^2(\Omega)]^3$ has zero average, i.e., $\strokedint_{\Omega}\vf=0$, for otherwise there is no equilibrium configuration. To see this, suppose that a non-zero external force $\vf$ is added to the discrete energy \eqref{e:Eh_intro} as well as to the right-hand side of the discrete gradient flow \eqref{gf:system}. Repeating the proof of Proposition \ref{ave} (evolution of averages), one can easily show that each step of the gradient flow yields
\begin{equation*}
\int_{\Omega}\delta\vy_{h }^{n+1}=\tau\int_{\Omega}\vf.
\end{equation*} 
Consequently, for $\|\delta\vy_{h }^{n+1}\|_{L^2(\Omega)}\to0$ as $n\to\infty$ it is necessary that $\int_{\Omega}\vf = 0$. In previous sections we assume, for simplicity of presentation, that $\vf=\mathbf{0}$ but the theory extends to $\int_{\Omega}\vf = 0$.

In this section, we prescribe Dirichlet boundary conditions on a portion $\Gamma^D\neq\emptyset$ of the boundary $\partial\Omega$, namely
\begin{equation} \label{def:BC_GammaD}
\vy=\vvarphi \quad \mbox{and} \quad \nabla\vy=\Phi \quad \mbox{on } \Gamma^D,
\end{equation}
where $\vvarphi\in[H^1(\Omega)]^3$ and $\Phi\in[H^1(\Omega)]^{3\times 2}$ is such that $\Phi^T\Phi=g$ a.e. in $\Omega$. 
In this case, we redefine the admissible set $\A$ as
\begin{equation} \label{def:admiss_BC}
\A:=\A(\vvarphi,\Phi):=\left\{\vy\in \V(\vvarphi,\Phi): \, \nabla\vy^T\nabla\vy=\g \quad \mbox{a.e. in } \Omega\right\},
\end{equation}
where $\V(\vvarphi,\Phi)$ is the affine manifold
\begin{equation} \label{def:space_BC}
\V(\vvarphi,\Phi):=\left\{\vy\in [H^2(\Omega)]^3: \, \restriction{\vy}{\Gamma^D}=\vvarphi, \, \restriction{\nabla\vy}{\Gamma^D}=\Phi\right\}.
\end{equation}
Furthermore, we also subtract the term $\int_{\Omega}\vf\cdot\vy$ from the energy $E(\vy)$ defined in \eqref{def:Eg_D2y}, where $\vf\in[L^2(\Omega)]^3$ is a given forcing function.
We resort to a Nitsche approach to impose the essential boundary conditions \eqref{def:BC_GammaD}. As a consequence, they do not need to be included as a strong constraint in the discrete counterpart of the admissible set $\A$.
This turns out to be an advantage for the analysis of the method \cite{bonito2018}.

Let $\Eh:=\Eh^0\cup\Eh^b$ be the set of edges of the subdivision $\Th$ decomposed into interior edges $\Eh^0$ and boundary edges $\Eh^{b}$. We assume that the Dirichlet boundary $\Gamma^D$ is compatible with $\Th$, $h>0$, in the sense that $\Gamma^D=  \{ \mathbf{x} \in e \ : e \in \mathcal \E_h^D \}$ for some $\E_h^D \subset \E_h^b$. The set of {\it active edges}, across which jumps and averages will be computed, and associated skeleton are denoted by $\Eh^a:=\Eh^0\cup\Eh^{D}$ and $\Gh^a := \Gh^0 \cup \Gamma^D$. For interior edges $e\in\Eh^0$, jumps and averages are defined (component-wise) by \eqref{def:jump-avrg}. For Dirichlet boundary edges $e\in\Eh^D$, we define averages by $\restriction{\avrg{\vv_h}}{e} := \vv_h$ and $\restriction{\avrg{\nabla_h\vv_h}}{e} := \nabla_h\vv_h$, while jumps are given by
\begin{equation}\label{E:bd-jumps}
\restriction{\jump{\vv_h}}{e} := \vv_h - \vvarphi,
\quad
\restriction{\jump{\nabla_h \vv_h}}{e} := \nabla_h \vv_h - \Phi.
\end{equation}
To simplify the notation, we define
\begin{equation}\label{discrete-set}
\V_h^k (\vvarphi,\Phi) := \Big\{ \vv_h\in [\V_h^k]^3: \
[\vv_h], \, [\nabla_h \vv_h] \text{ given by \eqref{E:bd-jumps} for all } e\in\Eh^D   \Big\}.
\end{equation}
We insist that $\V_h^k(\cdot,\cdot)$ coincides with $[\V_h^k]^3$  but the former carries the notion of boundary jumps.
In addition, for $\vv_h \in \V_h^k(\vvarphi,\Phi)$, $\|[\vv_h]\|_{L^2(\Gamma^D)} \to 0$ and $\|[\nabla_h \vv_h]\|_{L^2(\Gamma^D)} \to 0$ imply $\vv_h\to \vvarphi$ and $\nabla_h \vv_h \to\Phi$ in $L^2(\Gamma^D)$ as $h\to0$, thereby relating $\V_h^k(\vvarphi,\Phi)$ and $\V(\vvarphi,\Phi)$. 

The definitions \eqref{def:lift_re} and \eqref{def:lift_be} of lifting operators for interior edges $e\in\Eh^0$ extend trivially to boundary edges $e\in\Eh^D$, in which case $\omega_e$ reduces to a single element.
The discrete energy $E_h(\vy_h)$ then reads as \eqref{e:Eh_intro} upon (i) replacing $\Eh^0$ by $\Eh^a$ in the definition \eqref{E:global-lifting} of lifting operators, which affects the discrete Hessian operator \eqref{def:discrHess};
(ii) replacing $\Gamma_h^0$ by $\Gamma_h^a$ in the stabilization terms of $E_h$ and subtracting the forcing term $\int_{\Omega}\vf\cdot\vy_h$; and (iii) replacing the discrete admissible set by
	\begin{equation*}
	\A_{h,\veps}^k:=\{\vy_h\in \V^k_h(\vvarphi,\Phi)\ :\  D_h(\vy_h)\leq \veps\}.
	\end{equation*}
	Finally, we note that for non-homogeneous Dirichlet data, $|\cdot|_{H_h^2(\Omega)}$ defined in \eqref{e:H2semi} is no longer a semi-norm on $\V_h^k(\vvarphi,\Phi)$ since the boundary data are encoded in the jumps across $e\in\Eh^D$, but it is a norm on $\V^k_h(\bz,\bz)$ by definition.

All statements and proofs presented earlier extend to Dirichlet boundary conditions with minor modifications. To be concise, we summarize below the key differences between Dirichlet and \emph{free boundary} conditions.

\begin{enumerate}[$\bullet$]
  \medskip
		\item \textbf{Discrete Poincar\'e-Friedrichs inequality}:  
		For any $\vv_h\in\V_h^k(\vvarphi,\Phi)$ we have
		\begin{equation} \label{eqn:FP}
		\|\vv_h\|_{L^2(\Omega)}+\|\nabla_h\vv_h\|_{L^2(\Omega)}\lesssim|\vv_h|_{H_h^2(\Omega)}+\|\vvarphi\|_{H^1(\Omega)}+\|\Phi\|_{H^1(\Omega)};
		\end{equation}
		see \cite{bonito2018}. In contrast to \eqref{eqn:bound_noBC}, the term $\|\vv_h\|_{L^2(\Omega)}$ is not needed to bound $\|\nabla_h\vv_h\|_{L^2(\Omega)}$.
		
	    \medskip  
	    \item \textbf{Weak convergence of discrete Hessian}: Let $\{\vv_h\}_{h>0}\subset\V^k_h(\vvarphi,\Phi)$ satisfy $\|\vv_h\|_{H_h^2(\Omega)}\lesssim 1$ uniformly in $h$ and $\vv_h\rightarrow\vv$ in $L^2(\Omega)$ as $h\rightarrow 0$ for some $\vv\in[H^2(\Omega)]^3$. We proceed as in Lemma \ref{weak-conv} (weak convergence of $H_h$) given in Appendix \ref{sec:weak_strong} to prove that $H_h(\vv_h)\rightharpoonup D^2 \vv$ in $L^2(\Omega)$, except that integrating $\int_{\Omega}H_h(\vv_h):\vphi$ by parts twice gives the extra term
        \begin{equation*}\label{eqn:T6}
		T_6:=\sum_{e\in\Eh^D}\int_e(\vv_h-\vvarphi)\cdot\avrg{\di\mathcal{I}_h\vphi}\vn_e, \qquad \vphi\in[C_0^{\infty}(\Omega)]^{3\times 2\times 2}.
		\end{equation*}
        Its convergence is standard by uniform boundedness of $|\vv_h|_{H^2_h(\Omega)}$ and the trace inequality.

        \medskip  
	    \item \textbf{Strong convergence of discrete Hessian}: Let $\vv\in[H^2(\Omega)]^3$ satisfy $\vv=\vvarphi$ and $\nabla\vv=\Phi$ on $\Gamma^D$. The proof of Lemma \ref{strong-conv} (strong convergence of $H_h$) follows as in Appendix \ref{sec:weak_strong}.

        \medskip
	    \item \textbf{Coercivity}: The analogue of Theorem \ref{coercivity:plates} (coercivity of $H_h$) involves the boundary data and external forcing term, namely for any $\vy_h\in\V_h^k(\vvarphi,\Phi)$ and any $\gamma_0,\gamma_1>0$ we have
		\begin{equation} \label{eqn:coercivity_diri}
		|\vy_h|_{H_h^2(\Omega)}^2 \lesssim E_h(\vy_h)+\|\vvarphi\|_{H^1(\Omega)}^2+\|\Phi\|_{H^1(\Omega)}^2+\|\vf\|^2_{L^2(\Omega)};
		\end{equation}
		the proof is similar to \cite[Lemma 2.3]{bonito2018}. Note that now $E_h(\vy_h)$ is bounded from below but not necessarily by zero \cite{bonito2018}.

		\medskip
		\item \textbf{Compactness}: In contrast to Lemma \ref{l:compactness} (compactness), we do not need to consider a \emph{shifted} sequence with vanishing mean value. If a sequence $\{\vv_h\}_{h>0}\subset\V^k_h(\vvarphi,\Phi)$ satisfies $|\vv_h|_{H_h^2(\Omega)}\lesssim 1$, then there exists $\vv\in \V(\vvarphi,\Phi)$ such that, up to a subsequence, $\vv_h\to\vv$ in $[L^2(\Omega)]^3$ and $\nabla_h\vv_h\to\nabla\vv$ in $[L^2(\Omega)]^{3\times 2}$ as $h\rightarrow 0$. A proof of this statement follows along the lines of \cite[Proposition 5.1]{bonito2018}, thanks to \eqref{eqn:FP}. Therefore, Proposition \ref{prop:existence_yh} and Theorems \ref{conv_global_min} and \ref{lim-inf} are valid without removing the mean of $\vv_h$.

		\medskip  
		\item \textbf{Convergence of forcing term}: The addition of the forcing term in the energy does not affect Theorem \ref{lim-inf} (lim-inf of $E_h$) and Theorem \ref{lim-sup} (lim-sup of $E_h$) because $\int_{\Omega}\vf\cdot\vy_h\to\int_{\Omega}\vf\cdot\vy$ when $\vy_h\to\vy$ in $[L^2(\Omega)]^3$ as $h\to 0$. 

		\medskip
		\item \textbf{Gradient flow}: For \emph{free boundary} conditions, the gradient flow metric $\|\cdot\|_{H^2_h(\Omega)}$ contains an $L^2$ term to guarantee solvability of \eqref{gf:system} (see Remark \ref{kernel}) and control of the average of iterates (see Proposition \ref{ave}). In contrast, since $|\cdot|_{H_h^2(\Omega)}$ defined in \eqref{e:H2semi} is a norm on $\V_h^k(\mathbf{0},\mathbf{0})$ when Dirichlet boundary conditions are imposed, the extra $L^2$ term is no longer needed. The counterpart of the gradient flow of Section \ref{sec:GF} reads: given $\vy_h^0\in\A_{h,\veps_{0}}^k$ and $\tau>0$, iteratively compute $\vy_h^{n+1}:=\vy_h^{n}+\delta\vy_h^{n+1}\in \V_h^k(\vvarphi,\Phi)$ with $\delta\vy_h^{n+1}\in\mathcal{F}_h(\vy_h^n)$ satisfying
		\begin{equation} \label{gf:system_BC}
		\tau^{-1}\langle\delta\vy_h^{n+1},\vv_h\rangle_{H_h^2(\Omega)} +  a_h(\delta\vy_h^{n+1},\vv_h) = (\vf,\vv_h)_{L^2(\Omega)}-a_h(\vy_h^n,\vv_h) \qquad \forall\, \vv_h \in \mathcal{F}_h(\vy_h^n),
		\end{equation}
		where the tangent space is given by
         \[
         \mathcal{F}_h(\vy_h^n):=\left\{ \vv_h\in \V_h^k(\bz,\bz): \,\,
		\int_{\K}\nabla\vv_h^T\nabla\vy_h^n+(\nabla\vy_h^n)^T\nabla\vv_h=0 \quad \forall\, \K\in\Th\right\}
         \]
        and $\langle\cdot,\cdot\rangle_{H_h^2(\Omega)}$
		is defined in \eqref{def:H2bilinear}.
		Note that the Dirichlet data are implicitly contained in $a_h(\vy_h^n,\vv_h)$ through the liftings of the boundary data that appear in $H_h(\vy_h)$.
		
		For Dirichlet boundary conditions, the counterpart of the control of defect \eqref{eqn:control_defect} reads
		\begin{equation} \label{eqn:control_defect_diri}
		D_h(\vy_h^n)\le \veps_0+c\tau \big(E_h(\vy_h^0)+\wt c \big).
		\end{equation}
		Here $c>0$ is the hidden constant of \eqref{eqn:FP}, which depends only on $\Omega$ and $\Gamma^D$, while $\wt c\ge 0$ depends only on $\mu$, $g$, $\|\vvarphi\|_{H^1(\Omega)}$, $\|\Phi\|_{H^1(\Omega)}$, $\|\vf\|_{L^2(\Omega)}$ and the constant $C(\gamma_0,\gamma_1)$ that appears in \eqref{eqn:ineq_h2} (with $\Eh^0$ replaced by $\Eh^a$). The proof relies on \eqref{eqn:FP} and \eqref{eqn:coercivity_diri} to deal with the forcing term; the proof is similar to \cite[Lemma 3.4]{bonito2018}.
		
\medskip
\item \textbf{Preprocessing}: The prestrain defect of the iterates $\vy_h^n$ produced by \eqref{gf:system_BC} is controlled by \eqref{eqn:control_defect_diri}. In this case, the energy $E_h(\vy_h^0)$ is also affected by how well $\vy_h^0$ satisfies the prescribed boundary conditions as $E_h$ contains the terms $(\nabla_h\vy_h^0-\Phi)$ and $(\vy_h^0-\vvarphi)$ in the discrete Hessian and the stabilization terms. 
		Therefore, since flat surfaces are stationary for the \emph{metric preprocessing} step \cite[Section 3.3]{BGNY2020_comp}, we first solve the bi-Laplacian problem
		\begin{equation} \label{bi-Laplacian}
		\Delta^2\widehat \vy = \widehat\vf \quad \mbox{in } \Omega, \quad \nabla\widehat \vy = \Phi \quad \mbox{on } \Gamma^D, \quad \widehat \vy = \vvarphi \quad \mbox{on } \Gamma^D,
		\end{equation}
		where typically $\widehat\vf=\mathbf{0}$. Note that this vector-valued problem is well-posed and gives, in general, a non-flat surface $\widehat{\vy}(\Omega)$. Using the LDG method with boundary conditions imposed \emph{\`a la Nitsche} to approximate the solution $\widehat\vy\in\V(\vvarphi,\Phi)$ of \eqref{bi-Laplacian}, we thus solve:
		\begin{equation}\label{bi-Laplacian-system}
		\widehat \vy_h\in\V^k_h(\vvarphi,\Phi): \quad
		c_h(\widehat \vy_h,\vv_h) = (\widehat\vf,\vv_h)_{L^2(\Omega)} \qquad\forall\,\vv_h\in\V^k_h(\mathbf{0},\mathbf{0}).
		\end{equation}
		Here
		\begin{equation}\label{bi-Laplacian_bilinear}
		\begin{split}
		c_h(\vw_h,\vv_h) & := \big(H_h(\vw_h), H_h(\vv_h) \big)_{L^2(\Omega)}  \\
		&+\widehat\gamma_1\big(\h^{-1}\jump{\nabla_h\vw_h},\jump{\nabla_h\vv_h}\big)_{L^2(\Gh^a)}
		+\widehat\gamma_0 \big(\h^{-3}\jump{\vw_h},\jump{\vv_h}\big)_{L^2(\Gh^a)},
		\end{split}
		\end{equation}
		where $\widehat{\gamma}_0$ and $\widehat{\gamma}_1$ are positive stabilization parameters that may not necessarily be the same as their counterparts $\gamma_0$ and $\gamma_1$ used in the definition \eqref{e:Eh_intro} of $E_h$. Then $\wh\vy_h$ satisfies (approximately) the given boundary conditions on $\Gamma^D$ and $\wh\vy_h(\Omega)$ is, in general, non-flat. To generate a deformation with small prestrain defect, we may then apply a \emph{metric preprocessing} step similar to the one proposed in Section \ref{sec:preprocessing} starting from $\wt \vy_h^0=\wh \vy_h\in\V_h^k(\vvarphi,\Phi)$. Then all the results of Section \ref{sec:preprocessing} extend to Dirichlet boundary conditions upon replacing $\|\cdot\|_{H_h^2(\Omega)}$ by $|\cdot|_{H_h^2(\Omega)}$ and, wherever appropriate, $\Eh^0$ and $\Gh^0$ by $\Eh^a$ and $\Gh^a$, respectively. 
		In particular, the  boundary conditions satisfied by the initial deformation $\wt \vy_h^0\in\V_h^k(\vvarphi,\Phi)$ are approximately preserved during the \emph{metric preprocessing} step \eqref{pre-gf:variation2}. The latter consists of seeking increments $\delta \vy_h^n \in \V_h^k(\mathbf 0,\mathbf 0)$ minimizing the linearized stretching energy using the metric
		$$
(D^2_h v_h,D^2_h w_h)_{L^2(\Omega)}
  +(\h^{-1}\jump{\nabla_h v_h},\jump{\nabla_h w_h})_{L^2(\Gh^a)}+(\h^{-3}\jump{v_h},\jump{w_h})_{L^2(\Gh^a)}.
		$$
		
		Moreover, Remark \ref{rem:Eh_unif} (choice of $\sigma_h$) applies and leads to an output $\wt\vy_h^{n_h}\in\V^k_h(\vvarphi,\Phi)$ such that $E^b_h(\wt\vy_h^{n_h})\lesssim 1$ and $D_h(\wt\vy_h^{n_h}) \lesssim \sigma_h^{\frac12}$ provided that $E_h^p(\wt \vy_h^{n_h}) \lesssim \sigma_h$.
	\end{enumerate}
\section*{Acknowledgment}
Andrea Bonito and Diane Guignard were partially supported by the NSF Grant DMS-1817691 and DMS-2110811.
Ricardo H. Nochetto and Shuo Yang were partially supported by the NSF Grants DMS-1411808 and DMS-1908267.

\bibliographystyle{amsplain}        
\bibliography{abrevjournal,bibliography}

\providecommand{\bysame}{\leavevmode\hbox to3em{\hrulefill}\thinspace}
\providecommand{\MR}{\relax\ifhmode\unskip\space\fi MR }
\providecommand{\MRhref}[2]{%
  \href{http://www.ams.org/mathscinet-getitem?mr=#1}{#2}
}
\providecommand{\href}[2]{#2}
\begin{thebibliography}{10}

\bibitem{bangerth2007}
W.~Bangerth, R.~Hartmann, and G.~Kanschat, \emph{{deal.II} -- a general purpose
  object oriented finite element library}, ACM Trans. Math. Softw. \textbf{33}
  (2007), no.~4, 24/1--24/27.

\bibitem{bartels2013}
S.~Bartels, \emph{Finite element approximation of large bending isometries},
  Numer. Math. \textbf{124} (2013), no.~3, 415--440.

\bibitem{bonito2017}
S.~Bartels, A.~Bonito, A.H. Muliana, and R.H. Nochetto, \emph{Modeling and
  simulation of thermally actuated bilayer plates}, J. Comput. Phys.
  \textbf{354} (2018), 512--528.

\bibitem{bonito2015}
S.~Bartels, A.~Bonito, and R.H. Nochetto, \emph{Bilayer plates: Model
  reduction, {$\Gamma$}‐convergent finite element approximation, and discrete
  gradient flow}, Comm. Pure Appl. Math. \textbf{70} (2017), no.~3, 547--589.

\bibitem{bassi1997}
F.~Bassi, S.~Rebay, G.~Mariotti, S.~Pedinotti, and M.~Savini, \emph{A
  high-order accurate discontinuous finite element method for inviscid and
  viscous turbomachinery flows}, Proceedings of the 2nd European Conference on
  Turbomachinery Fluid Dynamics and Thermodynamics, 1997, pp.~99--109.

\bibitem{lewicka2016}
K.~Bhattacharya, M.~Lewicka, and M.~Sch{\"a}ffner, \emph{Plates with
  incompatible prestrain}, Arch. Rational Mech. Anal. \textbf{221} (2016),
  no.~1, 143--181.

\bibitem{BG21}
A.~Bonito and D.~Guignard, \emph{The step-82 tutorial program: solving the
  fourth-order biharmonic equation using a lifting operator approach}, deal.II
  library (2021), URL \url{https:
  //www.dealii.org/developer/doxygen/deal.II/step_82.html}.

\bibitem{BGNY2020_comp}
A.~Bonito, D.~Guignard, R.H. Nochetto, and S.~Yang, \emph{{LDG} approximation
  of large deformations of prestrained plates}, J. Comput. Phys. \textbf{448}
  (2022), 110719.

\bibitem{bonito2010quasi}
A.~Bonito and R.H. Nochetto, \emph{Quasi-optimal convergence rate of an
  adaptive discontinuous {G}alerkin method}, SIAM J. Numer. Anal. \textbf{48}
  (2010), no.~2, 734--771.

\bibitem{bonito2020discontinuous}
A.~Bonito, R.H. Nochetto, and D.~Ntogkas, \emph{Discontinuous {G}alerkin
  approach to large bending deformation of a bilayer plate with isometry
  constraint}, J. Comput. Phys. \textbf{423} (2020), 109785.

\bibitem{bonito2018}
\bysame, \emph{{DG} approach to large bending deformations with isometry
  constraint}, Math. Models Methods Appl. Sci. \textbf{31} (2021), no.~1,
  133--175.

\bibitem{BNS:21}
A.~Bonito, R.H. Nochetto, and S.~Yang, \emph{{LDG} approximation of large
  deformations of bilayer plates},  (In preparation).

\bibitem{brezzi1999}
F.~Brezzi, G.~Manzini, D.~Marini, P.~Pietra, and A.~Russo, \emph{Discontinuous
  finite elements for diffusion problems}, Atti Convegno in onore di F.
  Brioschi (Milano 1997), Istituto Lombardo, Accademia di Scienze e Lettere
  (1999), 197--217.

\bibitem{brezzi2000}
\bysame, \emph{Discontinuous {G}alerkin approximations for elliptic problems},
  Numer. Methods Partial Differential Equations \textbf{16} (2000), no.~4,
  365--378.

\bibitem{ciarlet1991}
P.G. Ciarlet, \emph{Basic error estimates for elliptic problems}, Handbook of
  {N}umerical {A}nalysis, {V}ol.\ {II}, Handb. Numer. Anal., II, North-Holland,
  Amsterdam, 1991, pp.~17--351.

\bibitem{cockburn1998}
B.~Cockburn and C.-W. Shu, \emph{The local discontinuous {G}alerkin method for
  time-dependent convection-diffusion systems}, SIAM J. Numer. Anal.
  \textbf{35} (1988), no.~6, 2440--2463.

\bibitem{efrati2009}
E.~Efrati, E.~Sharon, and R.~Kupferman, \emph{Elastic theory of unconstrained
  non-euclidean plates}, J. Mech. Phys. Solids \textbf{57} (2009), no.~4,
  762--775.

\bibitem{fortin2000}
M.~Fortin and R.~Glowinski, \emph{Augmented {L}agrangian methods: Applications
  to the numerical solution of boundary-value problems}, Elsevier Science,
  2000.

\bibitem{frie2002b}
G.~Friesecke, R.D. James, and S.~M{\"u}ller, \emph{A theorem on geometric
  rigidity and the derivation of nonlinear plate theory from
  three‐dimensional elasticity}, C.R. Math. \textbf{55} (2002), no.~11,
  1461--1506.

\bibitem{GR86}
V.~Girault and P.-A. Raviart, \emph{Finite element methods for
  {N}avier--{S}tokes equations: Theory and algorithms}, Springer Series in
  Computational Mathematics, vol.~5, Springer-Verlag, Berlin, 1986.

\bibitem{goriely2005differential}
A.~Goriely and M.~Ben Amar, \emph{Differential growth and instability in
  elastic shells}, Phys. Rev. Lett. \textbf{94} (2005), no.~19, 198103.

\bibitem{kim2012thermally}
J.~Kim, J.A. Hanna, R.C. Hayward, and C.D. Santangelo, \emph{Thermally
  responsive rolling of thin gel strips with discrete variations in swelling},
  Soft Matter \textbf{8} (2012), no.~8, 2375--2381.

\bibitem{klein2007shaping}
Y.~Klein, E.~Efrati, and E.~Sharon, \emph{Shaping of elastic sheets by
  prescription of non-euclidean metrics}, Science \textbf{315} (2007),
  no.~5815, 1116--1120.

\bibitem{modes2010disclination}
C.D. Modes, K.~Bhattacharya, and M.~Warner, \emph{Disclination-mediated
  thermo-optical response in nematic glass sheets}, Phys. Rev. E \textbf{81}
  (2010), no.~6, 060701.

\bibitem{modes2010gaussian}
\bysame, \emph{Gaussian curvature from flat elastica sheets}, Proc. Royal Soc.
  \textbf{467} (2010), no.~2128, 1121--1140.

\bibitem{ern2010}
D.A.~Di Pietro and A.~Ern, \emph{Discrete functional analysis tools for
  discontinuous {G}alerkin methods with application to the incompressible
  {N}avier--{S}tokes equations}, Math. Comp. \textbf{79} (2010), no.~271,
  1303--1330.

\bibitem{ern2011}
\bysame, \emph{Mathematical aspects of discontinuous {G}alerkin methods},
  Math{\'e}matiques et Applications, Springer Berlin Heidelberg, 2011.

\bibitem{pryer}
T.~Pryer, \emph{Discontinuous {G}alerkin methods for the p-biharmonic equation
  from a discrete variational perspective}, Electron. Trans. Numer. Anal.
  \textbf{41} (2014), 328 -- 349.

\bibitem{wu2013three}
Z.L. Wu, M.~Moshe, J.~Greener, H.~Therien-Aubin, Z.~Nie, E.~Sharon, and
  E.~Kumacheva, \emph{Three-dimensional shape transformations of hydrogel
  sheets induced by small-scale modulation of internal stresses}, Nat. Commun.
  \textbf{4} (2013), 1586.

\bibitem{yavari2010geometric}
A.~Yavari, \emph{A geometric theory of growth mechanics}, J. Nonlinear Sci.
  \textbf{20} (2010), no.~6, 781--830.

\end{thebibliography}

\end{document}